\documentclass[11pt]{article}
\usepackage{cite}

\usepackage{amsmath,amssymb,amsthm}
\usepackage{dsfont}
\usepackage{mathrsfs}
\usepackage{color}
\usepackage[utf8]{inputenc}
\usepackage[T1]{fontenc}
\usepackage{mathptmx}
\usepackage{changes}

\title{Quasi-processes for branching Markov chains}

\author{\textsc{Steffen Dereich and Martin Maiwald\blfootnote{Institute for Mathematical Stochastics, University of Münster, Orléans-Ring 10, 48149 Münster, steffen.dereich@wwu.de and martin.maiwald@me.com}}}

\DeclareMathAlphabet{\mathcal}{OMS}{cmsy}{m}{n}

\DeclareMathOperator{\pred}{pred}

\newcommand{\set}[1]{\{{#1}\}}

\newcommand{\dd}{\mathrm{d}}
\newcommand{\abs}[1]{\left| {#1} \right|}

\newcommand{\cB}{\mathcal{B}}
\newcommand{\cD}{\mathcal{D}}
\newcommand{\cT}{\mathcal{T}}
\newcommand{\cP}{\mathcal{P}}

\newcommand{\cI}{\mathcal{I}}

\newcommand{\cS}{\mathcal{S}}
\newcommand{\cU}{\mathcal{U}}
\newcommand{\cW}{\mathcal{W}}
\newcommand{\cK}{\mathcal{K}}

\newcommand{\IP}{\mathbb{P}}

\newcommand{\rd}{\color{red}}

\newcommand{\blue}{\mathrm{blue}}
\newcommand{\white}{{\mathrm{white}}}

\newcommand{\tr}{\mathfrak t}
\newcommand{\1}{\mathds{1}}
\newcommand{\N}{\mathbb{N}}
\newcommand{\Z}{\mathbb{Z}}
\newcommand{\E}{\mathbb{E}}

\newcommand{\IT}{\mathbb{T}}
\newcommand{\II}{\mathbb{I}}
\renewcommand{\P}{\mathbb{P}}

\newcommand{\cQ}{\mathcal Q}

\newcommand{\cH}{\mathcal H}
\newcommand{\rest}[2]{\left.{#1}\right|_{#2}}

\newcommand{\sfrac}[2]{\mbox{$\frac{#1}{#2}$}}
\newcommand{\eps}{\epsilon}

\newcommand{\bbar}[1]{{\bar {\bar #1}}}

\newcommand\blfootnote[1]{%
  \begingroup
  \renewcommand\thefootnote{}\footnote{#1}%
  \addtocounter{footnote}{-1}%
  \endgroup
}

\theoremstyle{definition}
\newtheorem{definition}{Definition}[section]

\newtheorem{lemma}[definition]{Lemma}            
\newtheorem{example}[definition]{Example} 
\newtheorem{remark}[definition]{Remark}

\newtheorem{construction}[definition]{Construction}
\newtheorem{corollary}[definition]{Corollary}

\newtheorem{notation}[definition]{Notation} 

\theoremstyle{theorem}
\newtheorem{theorem}[definition]{Theorem}    

\newtheorem{proposition}[definition]{Proposition}

\begin{document}

\maketitle

\noindent{\slshape\bfseries Abstract.} Potential theory is a central tool to understand and analyse Markov processes. In this article, we develop its probabilistic counterpart for branching Markov chains. 
Specifically, we examine versions of quasi-processes or interlacements that incorporate branching, referred to as branching quasi-processes. These processes are characterized by their occupation measures. If a certain decorability condition is fulfilled, there's an isomorphism between the set of branching quasi-processes and the set of excessive measures, where the excessive measures correspond to the occupation measures of the branching quasi-processes. Utilizing a branching quasi-process as an intensity measure for a Poisson point process leads to the formulation of random interlacements with branching. In cases where individuals reproduce with an average rate of one or less, we detail a construction that draws on classical interlacements. Our approach significantly employs the additive structure of branching Markov chains, with the spinal representation of the branching processes serving as a crucial technical tool.


\bigskip


\bigskip


\tableofcontents
\section{Introduction}

The study
 of potential theory and its probabilistic counterpart was a very prominent area of research in the last century, as evidenced by the number of monographs on this subject, for example, references \cite{Doob84,DS84,Get90}. 
 
This article concentrates on quasi-processes, which have recently become noteworthy for their role as intensity measures of random interlacements. A random interlacement, modeled as a Poisson point process, captures the local behavior of random walks near a reference point, relevant for understanding vacant set percolation, as shown in~\cite{SidSzn09,Szn10}.  

Originally, research was focused on random walks on large boxes of $\Z^{d}$, but similar concepts apply for other random walks such as random walks on random graphs, see~\cite{CTW11, CT12,CT13,CH22}.
Interlacements also help in characterizing local limits for limit theorems concerning random graphs, as seen in \cite{D23}.
Recently, a series of articles by Zhu addressed the analysis of visiting probabilities and the construction of interlacements for critical branching random walks on $\Z^d$,  
see \cite{Zhu1,Zhu2,Zhu3,Zhu4}. Connectivity properties are explored in \cite{ProZha19}.

In this article, we take a more generalized perspective, exploring the idea of a quasi-process for general branching Markov chains, which we term \emph{branching quasi-process}.
 In the particular case where the branching Markov chain (BMC) is a critical simple random walk on $\Z^{d}$ $(d=5,6,\ldots)$ one obtains a branching interlacement by generating a Poisson point process with the branching quasi-process being its intensity. In our analysis, the underlying BMC is very general and the concept of a  branching quasi-process is intuitively best described as a general branching process (possibly with infinite past) for which particles evolve forward in time as for classical BMCs.

We emphasise that our main result fully characterises this class of processes (under an appropriate decorability assumption) in terms of occupation measures.
 Our investigation paves a rigorous path to devising branching interlacements within significantly broader contexts, such as for branching random walks in heterogeneous environments. 
 It is a first step to showing local limit theorems as in \cite{D23} for branching random walks on random graphs. Additionally, we introduce a way to construct these branching interlacements using the framework of classical interlacements, effectively linking the two concepts.

From a technical point of view a BMC is per se a Markov chain so that it can be treated with the tools of potential theory where the classical state space, say $S$, needs to be replaced by the set of counting measures on $S$, where $S$ is now the  space in which the  constituents of the BMC are located. 
A key characteristic of a BMC is that its branches evolve independently, which simplifies the analysis and is essential for connecting branching quasi-processes to excessive measures on $S$, drawing on established results in the analysis of quasi-processes, see for instance \cite{Get90}.

The notion of a quasi-process stems from classical potential theory. It was introduced by Weil \cite{Weil71} as a
generalisation of its discrete version also called approximate Markov chain, studied by Hunt \cite{Hunt60}. We stick to the notion of a quasi-process even though we consider discrete time in this article. 
 We introduce the class of branching  quasi-processes and characterise the individual branching quasi-processes in terms of their occupation measure. Provided that a particular decorability assumption is satisfied we show that for every excessive measure there indeed exists a branching quasi-process which has the latter measure as occupation measure. This relation then provides  an isomorphism between  the cone of excessive measures and the cone of branching quasi-processes.
   Our analysis requires finiteness of certain Green's functions. Beyond that we only need to assume a weak form of irreducibility. The proofs  make strong use of a novel spine construction which is different from the ones used  by Zhu before.
   

Interestingly, sometimes in the construction of a branching quasi-processes, an explosion phenomenon can be observed. 
As found  by Zhu \cite{Zhu3} this is the case, for instance, for the simple, symmetric, critical branching random walk on $\Z^d$ for $d\le 4$ and for these processes an interlacement construction breaks down. More explicitly, in that case the ``decoration'' of a spine results in a configuration where all sites are visited by infinitely many individuals. We  will formalise the notion of decorability and give a sufficient criterion in terms of the Green's function. In particular, we find that the symmetric, strictly subcritical setting is always decorable.

We stress that our aim is to extend some parts of classical potential theory to branching Markov chains. The intensity measures of interlacements are quasi-processes. Such interlacements  appear as local limits when observing the traces left by a random walk around typical points in a large graph, a  crucial point towards understanding   vacant set percolation \cite{SidSzn09,Szn10,CTW11}. Results for branching random walks in that direction can be found in~\cite{Zhu4}. We believe that such connections hold in much more general situations and the current research aims to set the foundations for further analysis. 

In the next section, we introduce the central  notation and give one of the main results, see Theorem~\ref{thm:main3}. We also include a brief overview of the article at the end of this section.
In the article, we encode the structure of a BMC in a particular way, where each constituent/individual is equipped with an independent uniformly chosen label. This has the appeal that arguments can be carried out easily in full rigor and, in particular, the heavy use of the  branching Markov property is very transparent.

\section{Preliminaries and main results}\label{sec2}

We denote by $\N=\{1,2,\dots\}$ the set of strictly positive integers and set $\N_0=\N\cup\{0\}$.
We consider branching Markov chains. To introduce these we let
\begin{itemize}
	\item $S$ be a countable set (the \emph{state space}),
	\item $(p(x,y))_{x,y\in S}$ be a stochastic matrix (the \emph{transition kernel}), and
	\item $(d_x(m))_{x\in S, m\in\mathbb{N}_0}$ be a stochastic matrix (the \emph{offspring distribution}).
\end{itemize}

We consider branching Markov chains where individuals in a state $x\in S$ give independently rise to $m$ descendants with probability $d_x(m)$ each one having an independent $p(x,\,\cdot\,)$-distributed state.
%
%
%

One way to rigorously define a branching Markov chain is by using Ulam-Harris notation 
 where vertices are indexed by finite words with the empty word referring to  the initial individual.
   Generally, we call 
a set of finite words $\bigcup_{n=0}^\infty  \mathbb{N}^n$, where $\mathbb{N}^0= \set{\emptyset}$  and $\emptyset$ 
is  the empty word,  \emph{Ulam-Harris tree}, if it satisfies 
the following conditions:
\begin{enumerate}
	\item $U$ contains the empty word,
	\item discarding the rightmost letter of $u\in U\setminus\set{\emptyset}$ yields
	an element  $\hat{u}$ of $U$, 
	\item  reducing the last letter of a  word $u\in U$ 
 yields  again a word in $U$, and 
\item  for every length $n\in\N$ there are at most finitey many words of length $n$ in $U$.
\end{enumerate}
An Ulam-Harris tree $\cT$ encodes a genealogical tree with $\hat {u}$ being the predecessor of $u\in\cT\backslash\{\emptyset\}$.
We let $\IT=\IT(S)$ denote the set of all subsets $\mathfrak t$ of 
$$
\bigcup_{n=0}^\infty  \mathbb{N}^n \times S
$$
with the property that the projection onto the first component is injective and has as image an Ulam-Harris tree, say $\cP_{\mathfrak t}$. For $\mathfrak t\in \IT$, we denote by $V_{\mathfrak t}:\cP_{\mathfrak t}\to S$ the mapping with 
$$
(i,x)\in \mathfrak t \ \Leftrightarrow \ V_{\mathfrak t}(i)=x.
$$
 Further we set $\mathrm{pred}(i)=\hat i$ and denote for  $i\in\cP_{\mathfrak t}$ by
$$
C_{\mathfrak t}(i)=\{j\in \cP_{\mathfrak t}:\mathrm{pred}(j)=i\}.
$$
the descendants of $i$ in $\mathfrak t$.
We use the standard approach (see \cite{MR3803914}) to define the BMC$(d,p)$ in Ulam-Harris notation and write $\P^x_\mathrm{UH}$ for its distribution on $\IT(S)$ when started with an initial individual in $x\in S$. 

We define the Green's function $G=(g(x,y))_{x,y\in S}$ via 
$$
g(x,y)= \E^x\Bigl[\sum_{i\in \cP} \1 _{\{y\}} (V(i))\Bigr]
$$
for $x,y\in S$ which is just the expected  number of individuals in $y$ when starting the branching Markov chain in $x$. 
In the following, we will always assume that the Green's function is finite for all $x,y\in S$. The branching Markov chain is related to the linear operator $Q=(q_{x,y})_{x,y\in S}$ (to be called \emph{intensity operator}) defined by
\begin{align}\label{intensity_operator}
q_{x,y}= m_x p_{x,y},
\end{align}
where $m_x=\sum_{m\in\N} m \, d_x(m)$ is the expected number of descendants. As is well known one has
$$
G=\sum_{n=0}^\infty Q^n,
$$
where the matrix products $Q^n$ and also the infinite sum are well-defined since all entries of $Q$ are non-negative. In particular, by finiteness of the Green's function it follows finiteness of $m_x$ for every $x\in S$.

%

In oder to introduce the concept of a branching quasi-processes 
 we first need to introduce an appropriate notion for a branching process in doubly infinite time.
 Informally this is a random  finite or infinite set of (distinguishably) labeled vertices that have
  a location and one or none  predecessor.  
  Later we will use independent on $[0,1]$  uniformly chosen labels to index the constituents of the branching Markov chain and we denote by  $\mathcal U$ the uniform distribution on $[0,1]$.

\begin{definition}
Let $\tr\subset ([0,1]\cup\set{\emptyset})\times [0,1]\times S$ be a countable set.
 If 
\begin{itemize}
\item all points have a distinct second component (\emph{distinct labels}),
\item all points have in its first component either $\emptyset$ or the label of a different point (\emph{unique predecessor}),
\item no point is the iterated predecessor of itself (\emph{no circle condition}) and 
 \item for every $p\in [0,1]$, the set $\tr \cap (\set{p}\times [0,1]\times S)$ is
finite (\emph{transience}). 
\end{itemize}
then $\mathfrak t$ is called \emph{ordered forest}. 
\end{definition}

An ordered forest $\tr$ encodes a (possibly) doubly infinite genealogy
on the population
$$
\mathcal P_\tr=\{i\in [0,1]: \exists (p,x)\in ([0,1]\cup\set{\emptyset})\times S\text{ with } (p,i,x)\in \tr\}
$$
 which is the projecton of $\mathfrak t$ on the second component.
We set for $(p,i,x)\in\tr$
\[
  V_{\mathfrak{t}}(i) = x\quad \text{and}\quad  \pred_{\mathfrak{t}}(i) = p
\]
and call $V_{\mathfrak{t}}(i)$ the \emph{location} and  $\pred_{\mathfrak{t}}(i)$ the \emph{predecessor} of $i$. 
In the case that $\pred_{\mathfrak{t}}(i)=\emptyset$ we also say that $i$ does not have a predecessor. For convenience, we set $\pred_\mathfrak{t}(\emptyset)=\emptyset$ and $V_{\mathfrak t}(\emptyset)=\partial$ where $\partial$ is an additional cemetary state not contained in $S$. Note that~$\mathfrak t$ is uniquely described by the \emph{population} $\cP_\mathfrak{t}$, the \emph{predecessor mapping} $\pred_\tr:\cP_\tr\to \cP_\tr\cup\{\emptyset\}$ and the \emph{location mapping} $V_\tr:\cP_\tr\to S$.

An ordered forest describes a directed graph that is obtained  by linking all individuals of the population to their predecessor. As a consequence of the no circle condition the undirected version of the graph is indeed a forest.

\begin{definition}Let $\tr$ be an ordered forest. \begin{enumerate}
\item If  for every $x\in S$, the set $\tr\cap (([0,1]\cup\{\emptyset\})\times [0,1]\times \{x\})$ is finite, then $\tr$ is called \emph{transient}.
\item If the  forest $\tr$ is connected in the sense that arbitrary pairs of individuals have a common iterated predecessor, then $\tr$ is called \emph{ordered tree}. 
\end{enumerate}
We denote by $\mathfrak F(S)$ and $\mathfrak T(S)$ the set of transient, ordered forests and trees for the state space $S$, respectively. Typically, we omit the state space $S$ and briefly write $\mathfrak F$ and $\mathfrak T$.
\end{definition}

\begin{notation} We use the following notation when working with forests $\mathfrak t$. For a  set $B\subset S$ and $i\in \cP_\tr$ we set
\begin{enumerate}
\item  $\displaystyle{\mathfrak{F}_B=\bigl\{\tr\in\mathfrak F: \tr\cap\bigl(([0,1]\cup\{\emptyset\})\times[0,1]\times B\bigr) \not=\emptyset\bigr\}, \ \mathfrak T_B=\mathfrak F_B\cap \mathfrak T}$, \\
the set of \emph{transient forests/trees hitting $B$},
\item
$\displaystyle{
	\mathcal{H}_B(\tr) =
	\set{ (i,V_\tr(i)): i\in\mathcal P_\tr\text{ with  $V_t(i)\in B$  and $V_\mathfrak{t}(\pred^n(i))\not\in B$ for all $n\in\mathbb{N}$}},
}$ \\ the \emph{$B$-entrance of $\mathfrak t$},
\item 	$\displaystyle{
		\Pi_i(\mathfrak{t})=\set{(\emptyset,i,V(i))}\cup
		\set{(p,i',x)\in\mathfrak{t} :
		\exists n\in\mathbb{N}\text{ s.t.\ }
		\pred_{\mathfrak{t}}^{n}(i') = i},}$ \\ the \emph{subtree induced by the progeny  of $i$ in $\tr$},
\item $\displaystyle{
	\Pi_B(\mathfrak{t}) =\bigcup_{(i,x)\in\mathcal{H}_B(\tr)}\pi_i(\mathfrak{t}),}$ 
	the \emph{progeny of the $B$-entrance},
\item $\displaystyle{
  |i|_{\mathfrak{t}}
  = \sup\set{n\in\mathbb{N}_0: \pred_{\mathfrak{t}}^{n}(i)\not=\emptyset}}$, the \emph{generation of $i$}.
\end{enumerate}
\end{notation}

Ordered transient forests are formally represented as $\sigma$-finite counting measures and we endow the space~$\mathfrak F$ with the respective standard $\sigma$-field. 
We will use  a \emph{point process representation} for branching Markov chains. Roughly speaking, our notation is obtained from the classical   Ulam-Harris notation  by assigning each individual of a realization independently a  $\cU$-distributed label/identifier and forming a tree $\tr\in\mathfrak T$ by all tuples $(p,i,x)$ constituted by the identifier of the parent $p$, the identifier of the individual  $i$ and its location $x$.

Formally, we define a branching Markov chain in point process representation as follows:
\begin{definition}\label{def:BRW}
Let $\cI=\cI(S)$ be the set of finite subsets of $[0,1]\times S$ with distinct first components (labels).
For every  $\mathbf x\in \cI$ we denote by $\P^{\mathbf x}$ the distribution on $\mathfrak F$ that is formed as follows:
\begin{enumerate}
\item for each element $(i,x)$ of $\mathbf x$ we add one element $(\emptyset,i,x)$
 (generation $0$) and
\item then consecutively form generation by generation  by independently adding for each individual $(p,i,x)$ of the previous generation
\begin{enumerate}\item  a $d_x$-distributed number of descendants
\item at independent   $p(x,\cdot)$-distributed locations
\item with the predecessor being $i$ and
\item independent $\mathcal{U}$-distributed labels.
\end{enumerate}\end{enumerate}
We also define $\P^\mathbf{x}$ for finite counting measures $\mathbf x$ on $S$ by first assigning each point of~$\mathbf x$ an independent $\mathcal{U}$-distributed label and then proceeding as above. Furthermore, we briefly write for $x\in S$, $\P^x$ for $\P^{\delta_x}$ and we let for measures $\mu$ on $\cI$ or the finite counting measures on $S$
$$
\P^\mu= \int \P^\mathbf{x}\,\mu(\dd \mathbf x).
$$
We briefly call the respective branching Markov chain $\mathrm{BMC}(d,p)$.
\end{definition}

The aim of this article is to analyse the concept of a \emph{quasi-process} for branching Markov chains.
\begin{definition}\label{def:bqp}
A measure $\Xi$ on $\mathfrak T$ 
is called \emph{branching quasi-process} (for the $\mathrm{BMC}(d,p)$) if it satisfies for all finite subsets $B$ of $S$ that $\Xi( \mathfrak{T}_B)<\infty$ and that
\begin{align}\label{qp-prop}
	 \Xi|_{\mathfrak{T}_B}( \Pi_B\in \,\cdot\,) = \P^{\mu_B},
\end{align}
where $\mu_B$ is the related $B$-entrance measure  given by
$$
\mu_B= \Xi|_{\mathfrak F_B}(V(\cH_B)\in\,\cdot\,)
$$
with $V(\cH_B)$ being the counting measure of the locations of the individuals in $\cH_B$.
%
%
\end{definition}
\begin{remark}\begin{enumerate}
\item
A branching quasi-process is a measure under which $\Pi_B$ is for every finite set $B\subset S$ a branching Markov chain started in an appropriate distribution (finite measure) with all  individuals having independent $\cU$-distributed labels.  We call~(\ref{qp-prop}) the  \emph{Markov property} of the quasi-process. 
\item The labeling is independent of the related unlabelled  graph and hence the measure  $\Xi$ is uniquely specified by the distribution modulo graph isomorphisms (respecting the order). 
\item If almost every realization of a  branching quasi-process has an initial individual, then it equals $\P^\mu$ for a measure $\mu$ on~$S$. In what follows, we are mainly interested in branching quasi-processes without initial individual.
 \end{enumerate}
\end{remark}

To give the main result we need to additional definitions.

\begin{definition}
A finite subset  $B$ of $S$ satisfying  $$ G(x,B):=\sum_{y\in B} g(x,y)>0 \text{ \ \ or, equivalently, \ \ } \P^x(\cH_B\not=\emptyset)>0$$
for every $x\in S$, is called \emph{norming region}.
\end{definition}

\begin{theorem}\label{thm:main3}
Let $B\subset S$ be a norming region, set for every $x\in S$,  $h(x)=\E^x[\#\cH_B]$ 
and suppose that
$$
\bar m_x:=\sum_{k\in\N} (k-1)  \, d_x^\mathrm{s.b.}(k)
$$
is finite\footnote{This assumption is equivalent to the existence of the second moment of each offspring distribution $d_x$.}.  
If
\begin{align}\label{cond:deco2}
\sup_{x\in S} \frac 1{h(x)} \sum_{z\in B^c} \frac {\bar m_z}{m_z}  \, g(x,z) \,h(z)^2  <\infty,
\end{align}
then there exists  an  isomorphism between the cone of all branching quasi-processes and the cone of all $Q$-excessive measures such that a branching quasi-process $\Xi$ is related to a unique excessive measure $\nu$ on $S$ given by
$$
\nu(z)=\int \sum_{i\in \cP} \1_{\{z\}}(X(i)) \, d\Xi.
$$
\end{theorem}

The proof is based on a spine construction.  For a given norming region $B$ one considers the $h$-transformed Markov family $(P_h^x:x\in S)$ with state space  $S$ and transition probabilities
$$
p_{x,y}^h=\begin{cases}\frac 1{h(x)} q_{x,y} h(y), &\text{ if }x\in B^c,\\ 0,& \text{ if }x\in B.\end{cases}
$$
In Section~\ref{sec3}, we introduce the spine construction for branching Markov chains started in a point. More explicitly, we will show that a size-biased variant of the branching Markov chain can be obtained by first generating a $P_h$-chain and then applying a decoration procedure on the constituents of the latter chain, see Theorem~\ref{thm:spine} for the main result.

In Section~\ref{sec4}, we use the spine construction to characterise branching quasi-processes in terms of \emph{decorable} $P_h$-quasi-processes (Theorem~\ref{theo:alt_main}), where a $P_h$-quasi-process is called decorable if the decoration procedure of Section~\ref{sec3} produces an element of $\mathfrak T$ (and, in particular, satisfies the transience condition). A link of the occupation measures of the spine and the respective branching quasi-process is provided in Theorem~\ref{thm:35676}. 

Section~\ref{sec:potential_theory} is concerned with the potential theory of non-negative operators $Q$ that are not necessarily sub-Markovian. Here we characterise the cone of $P_h$-quasi-processes in terms of $Q$-excessive measures (Theorem~\ref{thm:main1}). Together with the findings of Section~\ref{sec4} we obtain a characterisation of branching quasi-processes in terms of excessive measures (Theorem~\ref{theo:57}).  

In Section \ref{sec6}, we derive simple construction mechanisms for branching quasi-processes or, equivalently, branching interlacements. If the operator $Q$ is sub-Markovian, then a branching interlacement can be constructed with the help of a classical $Q$-quasi-process with the same occupation measure, see Remark~\ref{rem:3578}.

In Section~\ref{sec7}, we discuss decorability of quasi-processes. We show that criterion~(\ref{cond:deco2}) implies that all $P_h$-quasi-processes are decorable (Theorem~\ref{thm:485}). Together with Therem~\ref{theo:57} it follows Theorem~\ref{thm:main3} above.   We provide further sufficient criteria in the case where $Q$ is symmetric, see Theorem~\ref{thm:deco2}. In particular, it follows that the strictly subcritical setting is decorable, see Remark~\ref{rem946567}.


%
%
%

\section{The $B$-biased branching Markov chain}\label{sec3}

We will make strong use of a related branching Markov chain, the so-called  \emph{$B$-biased BMC}. This process admits a very simple and useful spine construction that will be central in the proof of our main results. 

We denote for $x\in S$
$$
h(x)= \E^x[\#\cH_B]
$$
which is finite for all $x\in S$ since the Green's functions are assumed to be finite.
Note that $ h$ solves 
$$
 h(x)=\begin{cases} Q h(x), & x\not \in B,\\
1, & x\in B.
\end{cases}
$$
%
%
%


\begin{definition} 
The $B$-biased BMC$(d,p)$ is a coloured (or multitype) BMC with particles of two colours blue and white, i.e., with state space $\bar S=S\times \{\mathrm{white}, \mathrm{blue}\}$. It evolves as follows:
\begin{itemize}
\item A white particle produces only white descendants according to the original BMC$(d,p)$ rule.
\item A blue particle in $B$ produces offspring as a white particle does.
\item A blue particle in $x\not\in B$ produces $n\in\N$ descendants with locations $y_1,\dots,y_n\in S$ with probability
\begin{align}\label{bias-BMC-weights}
\frac 1{h(x)}\, d_x(n) \,\prod_{k=1}^n p_{x,y_k} \,\sum_{k=1}^n  h(y_k).
\end{align}
If $n$ descendants with locations $y_1,\dots,y_n$ have been generated in the first step, one  marks one descendant blue and all others white with the $k$-th descendant being chosen with conditional probability
$$
\frac { h(y_k)}{\sum_{\ell=1}^n h(y_\ell)}.
$$
\end{itemize}
We let $\bar \cI=\cI(\bar S)$ be the set of finite subsets of $[0,1]\times \bar S$ with distinct first components (labels) 
and denote for $\mathbf x\in\bar \cI$ by $\bar{\P}_B^{\mathbf x}$ the distribution on $\bar{ \mathfrak F}=\mathfrak F(\bar S)$ with initial population~$\mathbf x$ that is generated in analogy  to Definition~\ref{def:BRW} according to the above rules.
We use analogous notation as in Definition~\ref{def:BRW} and, in particular, denote by $\bar \IP^{(x,c)}_B$ the distribution that is generated when starting with one particle of colour $c\in\{\white,\blue\}$ at position $x\in S$.
Note that the blue particles form a random chain, \emph{the $B$-biased spine}, that ceases to exist after entering $B$.
\end{definition}

\begin{remark}Note that (\ref{bias-BMC-weights}) indeed defines a sequence of probability weights since $q_{x,y}=m_x p_{x,y}$ and $Qh(x)= h(x)$ ($x\not\in B$) imply that
\begin{align*} 
\sum_{n\in\N}&\sum_{y_1,\dots,y_n\in\Sigma} \frac 1{\bar h(x)}\, d_x(n) \,\prod_{k=1}^n p_{x,y_k} \,\sum_{k=1}^n \bar h(y_k) \\
& =\sum_{n\in\N} \frac {nd_x(n)}{m_x} \underbrace{\sum_{y_1,\dots,y_{n-1}\in S} \prod_{k=1}^{n-1} p_{x,y_k}}_{=1} \underbrace{ \frac {1}{\bar h(x)} \sum_{y} q_{x,y}\bar h(y)}_{=1}=1.
\end{align*}
\end{remark}
We give an alternative description of the branching rule of the $B$-biased BMC.

\begin{proposition}\label{prop35783}For the $B$-biased BMC a blue individual in $x\not \in B$ produces a size biased number of descendants, meaning that $n$ descendants are produced with probability
$$
d_x^\mathrm{s.b.}(n):=\frac {n d_x(n)}{m_x},
$$
with the $n-1$ white descendants  evolving independently  according to $p$ and the blue descendant evolving according to $(p^h_{x,y})$ with
\begin{align}\label{def:ph}
p_{x,y}^{ h}=\1_{B^c}(x)\, \frac {q_{x,y} h(y)}{ h(x)}.
\end{align}
\end{proposition}

We denote by $(P_h^x:x\in S)$ the sub-Markov family for $(p^h_{x,y})_{x,y\in S}$ given by~(\ref{def:ph}).
Proposition~\ref{prop35783} entails that the path of the blue individuals under  $\bar \P^{x,\mathrm{blue}}_B$ is  $P_h^x$-distributed.

\begin{proof}
Let $n\in\N$, $y_1,\dots,y_n\in S$ and denote by $\cD_{y_2,\dots,y_n}^{y_1}$ the event that the first generation is constituted by one blue individual in $y_1$ and $n-1$ individuals in $y_2,\dots,y_{n}$ (with the white individuals being ordered lexicographically). One has
\begin{align*}
\bar \IP_B^{x,\mathrm{blue}}(\cD_{y_2,\dots,y_n}^{y_1}) &=  n \, \frac 1{ h(x)}\, d_x(n) \,\prod_{k=1}^{n} p_{x,y_k} \,\sum_{k=1}^n  h(y_k) \, \frac { h(y_1)}{\sum_{k=1}^n h(y_k)}\\
&= \frac {nd _x(n)}{m_x} \, \underbrace{\frac{q_{x,y_1}  h(y_1)}{ h(x)}}_{=p_{x,y_1}^h}\, \prod_{k=2}^{n} p_{x,y_k} .
\end{align*}
\end{proof}

The spine construction of a $B$-biased BMC uses the following decoration procedure.


\begin{construction}[The $B$-biased decoration] \label{const:dec2}
We describe a probability kernel that associates a quasi-path  $w^*=[w]\in D^*$ 
with a distribution $\bar \Gamma^\mathrm{deco}(w^*,\,\cdot\,)$ on the set of ordered coloured trees (possibly with the transience condition  not being satisfied). A $\bar \Gamma^{\text{deco}}(w^*,\,\cdot\,)$-distributed random tree is obtained as follows:
\begin{enumerate}
\item first generate the linear tree (the spine) associated with $w^*=[w]$ according to the following rules: 
 for every $k\in \mathbb Z\cap [ T_\partial(w),  T^\partial(w))$ we add a blue individual with independent $\cU$-distributed identifier $i_k$ at position $w_k$ with predecessor $i_{k-1}$ (and with no predecessor if $k=T_\partial(w)$);
\item second attach to every individual $(p,i,x)$ of the spine with $x\not\in B$ (typically all individuals of the spine except the last one)  independently  with probability $d_x^\mathrm{s.b.}(n)$, $n-1$ immediate white descendants at independent $p(x,\cdot)$-distributed locations (with independent $\cU$-distributed labels) and then attach independent white  BMC($d,p$)'s to each of the latter descendants;
\item third attach to every  individual $(p,i,x)$ of the spine  with $x\in B$ (typically the last individual of the spine) an independent white  BMC($d,p$).
\end{enumerate}
The resulting tree can be conceived uncoloured in which case we refer to the probability kernel by $\Gamma^\mathrm{deco}$.
\end{construction}

\begin{proposition}\label{prop:35982}
Let $x\in S$. When applying the $B$-biased decoration on a $P_{h}$-Markov chain started in $x$, the resulting random tree is $\bar \P_B^{x,\mathrm{blue}}$-distributed. In particular,  $P_h^x$-almost every path enters $B$.
\end{proposition}
The proposition is an immediate consequence of the definition of the $B$-biased BMC and Proposition~\ref{prop35783}.

%

For $\mathbf x\in \cI$,  we denote by $ \P_B^{\mathbf x}$ the \emph{$B$-size biased analogue of $\P^\mathbf{x}$} that is given by
$$
\frac{\dd  \P^\mathbf {x}_B}{\dd \P^\mathbf{x}}= \frac 1{\E^{\mathbf x}[\# \cH_{B}]}\,  \# \cH_{B}.
$$
We relate $\bar \P^{x,\mathrm{blue}}_B$ and $\P^x_B$.
To do so we use $B$-colourings as introduced in the following construction.
\begin{construction}[The $B$-biased colouring]
The $B$-biased colouring is a probability kernel from the set $\mathfrak F_B$ to the set $\bar{\mathfrak F}_B$ of blue and white coloured forests. For given $\tr\in\mathfrak F_B$, $\bar \Gamma^{\text{col}}(\tr,\cdot)$ is the distribution obtained by picking an individual in $\cH_B(\tr)$ uniformly at random and marking the individual itself and  its ancestral line  in blue and all other individuals in white.
\end{construction}

\begin{theorem}\label{thm:spine}
For  $x\in S$, we have
$$
\bar \P_B^{(x,\mathrm{blue})}=\P_B^x\otimes \bar \Gamma^\mathrm{col}= P_h^x\otimes \bar \Gamma^\mathrm{deco}.
$$
Moreover, for $\mathbf x \in \cI$, we have
$$
\bar \P^{\mathrm{col}_h(\mathbf{x})}_B = \P_B^\mathbf {x} \otimes \bar \Gamma^{\mathrm{col}},
$$
 where $\mathrm{col}_h(\mathbf{x})$ is the distribution on $\bar\cI$ obtained by a random colouring of $\mathbf x$ that chooses  an element $(i,x)$ of $\mathbf x$ with probability proportional to $h(x)$ and marks it blue and all other entries white. 
\end{theorem}

When applying a $B$-biased colouring on a $\P_B^{\mathbf{x}}$-distributed forest the locations of the blue individuals (spine) form a $P_h$-Markov chain. Given the spine the conditional distribution of the tree connected to the spine is given by the $B$-biased decoration kernel as consequence of Prop.~\ref{prop:35982}.

\begin{proof}
We note that the equality $\bar \P_B^{(x,\mathrm{blue})}=  P_h^x\otimes \bar \Gamma^\mathrm{deco}$ is just the statement of Prop.~\ref{prop:35982}.
We now prove validity of $\bar \P_B^{(x,\mathrm{blue})}=\P_B^x\otimes \bar \Gamma^\mathrm{col}$. We use Ulam-Harris notation 
and    write $\bar \P_{\mathrm{UH},B}^{(x,\mathrm{blue})}$ and  $ \P_{\mathrm{UH},B}^x$ for the Ulam-Harris variants  of~$\bar \P_B^{(x,\mathrm{blue})}$ and $\P_B^{x}$. More explicitly, $\bar \P_{\mathrm{UH},B}^{(x,\mathrm{blue})}$- and  $\P_{\mathrm{UH},B}^x$-distributed objects are obtained by taking $\bar \P_B^{(x,\mathrm{blue})}$- and $\P_B^{x}$-distributed trees,  enumerating the children of each individual independently and uniformly at random and forgetting about the original labeling.
  

We call an element  $\bar \tr$ of $\IT(S\times\{\mathrm {blue, white}\})$  $n$-feasible if it satisfies the  following:  $\bar \tr$~has at most $n$ generations, the initial generation consists of  one blue individual, every blue individual in the generations $0,\dots,n-1$ outside of $B$ has exactly one blue descendant. Every blue individual in $B$ and all white individuals have only white descendants. Moreover, we denote by $\cT$, resp.\ $\bar \cT$, the canonical random tree on the respective uncoloured and coloured probabiliy spaces and let $\cT^B$ denote a $B$-biased colouring of $\cT$ which is obtained by independently choosing an element of $\cH_B(\cT)$ and marking itself and its ancestoral line in blue. 

We proceed by induction. Suppose we showed already that  for $n\in\N$ one has for $x\in S$ and all $n-1$-feasible $\bar \tr$, $ \P^{ x}_{\mathrm{UH},B}(\cT^B|_{n-1}=\bar\tr)=\bar \P_{\mathrm{UH},B}^{x,\mathrm{blue}}(\bar \cT|_{n-1}=\bar\tr)$. 
We show validity for $n-1$ replaced by $n$.

Suppose that $\bar \tr$ is an $n$-feasible configuration with one initial blue individual in $x\in B^c$ having $n$ descendants in locations $y_1,\dots,y_n$ with the $j$th individual being blue. 
We denote by $\tr$ the uncoloured version of $\bar\tr$ and by  $\Pi_i$ the subtree generated by the $i$th individual   and its descendants. One has
\begin{align*}
 \P_{\mathrm{UH},B}^x(\cT^B|_n=\bar \tr)&=\frac 1{h(x)} \E^x_\mathrm{UH}[\#\cH_B\, \1\{\cT^B|_n=\bar \tr\}] \\
&\stackrel{(a)}= \frac 1{h(x)} d_x(n) \prod_{i=1}^n p_{x,y_i}\,\E^{(y_1,\dots,y_n)}_\mathrm{UH}\Bigl[ \#\cH_B(\Pi_j) \1\{(\Pi_j (\cT))^B|_{n-1}=\Pi_j(\bar  \tr)\} \\
& \hspace{5.6cm} \prod_{i\not= j} \1\{\Pi_i (\cT)|_{n-1}=\Pi_i( \tr)\}\Bigr] \\
&\stackrel{(b)}= \frac 1n  d_x^\mathrm{s.b.}(n)\, p^{\bar h}_{x,y_j}  \frac 1{\bar h(y_j)} \E^{y_j}_\mathrm{UH}\bigl[ \#\cH_B \1\{\cT^B|_{n-1}=\Pi_j(\bar \tr)\}\bigr]\\
& \hspace{5cm} \prod_{i\not= j} p_{x,y_i}\, \P^{y_i}_\mathrm{UH}\bigl(\cT|_{n-1}=\Pi_i( \tr)\bigr)\\
&\stackrel{(c)}= \frac 1n  d_x^\mathrm{s.b.}(n) \,p^{\bar h}_{x,y_j}\,   \bar\P^{y_j,\mathrm{blue}}_{\mathrm{UH},B}\bigl( \bar \cT|_{n-1}=\Pi_j(\bar \tr)\bigr) \\
& \hspace{5cm} \prod_{i\not= j} p_{x,y_i}\, \bar \P^{y_i,\mathrm{white}}_{\mathrm{UH},B}\bigl(\bar \cT|_{n-1}=\Pi_i( \bar \tr)\bigr)\\
& = \bar \P_{\mathrm{UH},B}^{x,\mathrm{blue}}(\bar \cT|_n=\bar \tr).
\end{align*}
Here we used in step $(a)$ that for the $B$-biased colouring an individual in the branch  $\Pi_j(\cT)$ is picked with probability $\#\cH_B(\Pi_j)/\#\cH_B$ and the conditional distribution of the picked individual given that the $j$th branch is chosen is  uniform over $\cH_B(\Pi_j)$.
In step $(b)$ we used that branches of the BMC evolve independently and the definitions of $d_x^{\mathrm{s.b.}}(n)$ and $p^{\bar h}$, and in $(c)$ we used the induction hypothesis. Thus we showed equality for $n$-feasible trees with initial individual in $x\in B^c$. In the case where we start with an individual in $B$ the equality is trivial. Moreover, in the case where $\tr$ is not feasible or does not have an initial individual in $x$ both sides   equal zero.

It remains to consider the case where we start with finitely many individuals. 
Let  ${\mathbf x}=(x_1,\dots,x_m)\in S^m$ be a vector of $m$  individuals. Suppose that $\tr\in \IT(S\times \{\mathrm{blue},\mathrm{white}\})$ has at most $n$ generations and that the vertices of the initial generation agree up to the colouring with $\mathbf x$. Moreover, suppose that the $j$th initial individual of $\tr$ is blue and all others are white. We denote by $\cB_j$ the event that the $j$th initial individual is coloured in blue.
One has
\begin{align*}
 \P_{\mathrm{UH},B}^{\mathbf x}( \cT^B|_n= \mathfrak t)&=\frac 1{\sum_i  h(x_i)}  \E_\mathrm{UH}^{\mathbf x}[\# \cH_B\, \1\{   \cT^B|_n= \mathfrak t\}]\\
&= \frac 1{\sum_i  h(x_i)}  \E_\mathrm{UH}^{\mathbf x}\Bigl[\# \cH_B\, \1_{\cB_j} \1\{   \Pi_j(\cT)^B|_n= \Pi_j(\mathfrak t)\} \\
&\hspace{3cm}\prod_{i\not=j} \1\{   \Pi_i(\cT)^\mathrm{white}|_n= \Pi_i(\mathfrak t)\}\Bigr]\\
&\frac 1{\sum_i h(x_i)}  \E^{\mathbf x}_\mathrm{UH}\Bigl[\# \cH_B(\Pi_j(\cT))  \1\{   \Pi_j(\cT)^B|_n= \Pi_j(\mathfrak t)\} \\
&\hspace{3cm}\prod_{i\not=j} \1\{   \Pi_i(\cT)^\mathrm{white}|_n= \Pi_i(\mathfrak t)\}\Bigr]\\
&=\frac {h(x_j)}{\sum_{i} \bar h(x_i)} \bar \P_{\mathrm{UH},B}^{x_j,\mathrm{blue}}(\bar \cT|_n=\Pi_j(\tr)) \, \prod_{i\not= j} \bar \P_{\mathrm{UH},B}^{x_i,\mathrm{white}}(\bar\cT|_n=\Pi_i(\tr)).
\end{align*}
Note that indeed the initial blue individual is chosen proportional to the $h$-value and conditonal on the coloured initial population the process is a $B$-biased BMC. 
%
\end{proof}

\section{Characterisation in terms of the cone of decorable $P_h$-quasi-processes}\label{sec4}

In this section we characterise branching quasi-processes in terms of decorable $P_h$-quasi-processes.

\begin{definition}
We call a measure $\eta$ on $D^*$ \emph{decorable}, if $\eta\otimes \bar \Gamma^{\mathrm{deco}}$ is supported on $\bar{\mathfrak T}$.
\end{definition}
\begin{notation}
Let $\mathfrak t\in \bar{\mathfrak F}_B$ with the blue vertices forming a linear subtree\footnote{A linear tree is a tree where each individual has either one or zero descendants.} with ordered indices $(i_k)$. Then we denote the quasi-paths of the blue vertices by 
$$
\mathrm{spine}(\mathfrak t) = [(V_{\mathfrak t}(i_k))_k].
$$
\end{notation}

\begin{theorem}\label{theo:alt_main} Let $B$ be a norming region. 
There is a one-to-one isomorphism between the cone of all branching quasi-processes $\Xi$ and the cone of all decorable $P_{ h}$-quasi-processes $\eta$ satisfying 
\begin{align}\label{eq:equiv}
\eta=(\Xi_B\otimes \bar \Gamma^\mathrm{col})\circ \mathrm{spine}^{-1} \text{ \ \ and \ \ } \Xi_B\otimes \bar \Gamma^\mathrm{col} =\eta\otimes \bar \Gamma^{\mathrm{deco}},
\end{align}
where $\Xi_B$ denotes the respective  $B$-biased analogue of $\Xi$, i.e.,
$$
\frac{d\Xi_B}{d\Xi}= \#\cH_B.
$$
In particular, $\Xi_B$ is a finite measure.
\end{theorem}

\begin{proof}
1) We start with a $B$-biased version $\Xi_B$ of a branching quasi-process $\Xi$ and show that $\eta:=(\Xi_B\otimes \bar \Gamma^\mathrm{col})\circ \mathrm{spine}^{-1}$ is a $P_{h}$-quasi-process.
For $B'\supset B$ finite one has
\begin{align*}
\eta(\pi_{B'}\in\,\cdot\,)&=\int  \1\{\pi_{B'}(\mathrm{spine})\in \,\cdot\,\}\,   d\Xi_B \otimes  \bar \Gamma^{\text{col}}\\
&= \int  \#\cH_B \,  \bar \Gamma^{\text{col}}(\Pi_{B'}, \mathrm{spine}\in \,\cdot\,)\,  d \Xi  \\
&= \int \int \#\cH_B(\tr)\, \bar \Gamma^{\text{col}}(\tr, \mathrm{spine}\in \,\cdot\,)   \, d \P ^{\mathbf x} (\tr) \,d\mu_{B'}(\mathbf x)\\
 &= \int h(\mathbf x) \int \ \bar \Gamma^{\text{col}}(\tr, \mathrm{spine}\in \,\cdot\,)   \, d \P_B ^{\mathbf x} (\tr) \,d\mu_{B'}(\mathbf x)\\
 &=   \int \sum_{(i,x)\in \mathbf x} h(x) \int \ \bar \Gamma^{\text{col}}(\tr, \mathrm{spine}\in \,\cdot\,)   \, d \P_B ^{x} (\tr) \,d\mu_{B'}(\mathbf x),
\end{align*}
where $\mu_{B'}$ is the $B'$-entrance measure of $\Xi$. By Theorem~\ref{thm:spine}, $\P^x_B\otimes \bar \Gamma^\mathrm{col}=P_h^x\otimes \bar \Gamma^{\mathrm{deco}}$ so that, in particular, the spine is $P_h^x$ distributed under the latter distribution. Hence, we get that
$$
\eta(\pi_{B'}\in\,\cdot\,)= \int P_{ h}^x  \,d\bar \mu_{B'}(x),
$$
where $\bar \mu_{B'}$ is the finite measure on $S$ given by
$$
\bar \mu_{B'}(A)= \int  \sum_{(i,x)\in \mathbf x} h(x)\, \1_A(x) \, d\mu_{B'}(\mathbf x).
$$
We proved that $\eta$ is a $P_{ h}$-quasi-process with $B'$-entrance measure $\bar \mu_ {B'}$. Moreover, since all paths enter $B$, we get that 
$$
\Xi_B(\mathfrak T)= \eta(D^*) =\eta(D^*_B)= \bar \mu_{B'}(B')
$$
is, in particular, finite.

To verify the second equality in~(\ref{eq:equiv}) we fix finite sets $B_n\subset S$ with $B\subset B_1$ and $B_n\uparrow S$ and we define, for a coloured tree $\bar\tr\in \bar{\mathfrak T}_B$ with a unique blue spine hitting $B$,
$$
\mathfrak s_n(\bar\tr)= \text{ the first individual on the spine in $B_n$}. 
$$
Then
\begin{align*}
\int \1\{\Pi_{\mathfrak s_n} \in \,\cdot\,\} \, d\Xi_B\otimes \bar \Gamma^\text{col} &= \int \#\cH_B \,\1\{ \Pi_{\mathfrak s_n} \in \,\cdot\,\}\, d\Xi\otimes \bar \Gamma^\text{col}\\
&= \int \int \#\cH_B \,\1\{ \Pi_{\mathfrak s_n} \in \,\cdot\,\}\, d\P^{\mathbf x}\otimes  \bar \Gamma^\text{col} \ d\mu_{B_n}(\mathbf x)\\
&=\int  \P_B^x\otimes \bar \Gamma^\text{col} (\,\cdot\,) \,d\bar\mu_{B_n}(x) = \int P^x_h\otimes \bar \Gamma^{\mathrm{deco}} (\,\cdot\,) \,d\bar\mu_{B_n}(x)\\
&=\int \bar \Gamma^\mathrm{deco}(\pi_{B_n},\,\cdot\,)\,      d\eta.
\end{align*}
This implies that for arbitrarily fixed finite set $B'\supset B$ one has
\begin{align*}
\int \1\{\Pi_{B'}\in \,\cdot\,\}\, d\Xi_B\otimes \bar \Gamma^\mathrm{col}&=\lim_{n\to\infty} \int \1\{\Pi_{B'}\circ \Pi_{\mathfrak s_n} \in \,\cdot\,\} \, d\Xi_B\otimes \bar \Gamma^\text{col}\\
&=\lim_{n\to\infty}  \int \bar \Gamma^\mathrm{deco}(\pi_{B_n}, \Pi_{B'}\in \,\cdot\,)\,      d\eta\\
&=\int\bar \Gamma^\mathrm{deco}(w,\Pi_{B'}\in \,\cdot\,)\,      d\eta([w]).
\end{align*}
and we showed validity of the second equality in~(\ref{eq:equiv}). In particular, $\eta$ is decorable.

2) It remains to construct  for a decorable   $P_h$-quasi-process~$\eta$  a branching quasi-process $\Xi$ so that its $B$-biased variant satisfies~(\ref{eq:equiv}). We set $\bbar \Xi=\eta\otimes \bar \Gamma^\mathrm{deco}$ and denote by~$\bar \Xi$ the distribution $\bbar \Xi$ after removal of the colouring. 
Next, we verify the two  equalities of~(\ref{eq:equiv}) for $\bar \Xi$ in place of $\Xi_B$.

We start with the second equality.
Let $B',B_1,B_2,\dots$ be  finite subsets of $S$ with $B'\supset B$ and $B_n\uparrow S$.
By definition of the decoration and colouring procedure and Theorem~\ref{thm:spine} one has in terms of the entrance measures $(\bar \mu_{B_n})$ of $\eta$:
\begin{align}\begin{split}\label{eq83256}
\bbar \Xi ( \Pi_{B'}\in \,\cdot\,)&=\lim_{n\to\infty}
((\eta\circ (\pi_{B_n})^{-1}) \otimes \bar \Gamma^{\mathrm{deco}})(\Pi_{B'}\in\,\cdot\, )
\\
&=\lim_{n\to\infty}\int  (P_h^{x}\otimes \bar \Gamma^\mathrm{deco})(\Pi_{B'}\in \,\cdot\,)\, d\bar \mu_{B_n}(x)\\
&
= \lim_{n\to\infty}\int  (\P_B^{x}\otimes \bar \Gamma^\mathrm{col})(\Pi_{B'}\in \,\cdot\,)\, d\bar \mu_{B_n}(x) = ( \bar \Xi\otimes \bar \Gamma^\mathrm{col})(\Pi_{B'}\in \,\cdot\,),
\end{split}\end{align}
where we used that the action of the colouring procedure on $\Pi_{B'}$ depends only on~$\Pi_{B'}$ and that $\bar \Xi(\bar \Pi_{B'}\in\,\cdot\,)=\lim_{n\to\infty} \int  \P_B^{x}(\Pi_{B'}\in \,\cdot\,)\, d\bar \mu_{B_n}(x)$. Thus we showed the second equality. 
The first equality follows immediately since $\bar \Xi\otimes \bar \Gamma^\mathrm{col}=\bbar \Xi =\eta\otimes \bar\Gamma^\mathrm{deco}$.  

It remains to construct a branching-quasi-process $\Xi$ whose $B$-biased variant satisfies~$\Xi_B=\bar \Xi$.
Again using~(\ref{eq83256}) we obtain that, for $B'\supset B$ finite,
$$
\bar \Xi ( \Pi_{B'}\in \,\cdot\,)=\lim_{n\to\infty}  \int\P_B^{x}(\Pi_{B'}\in\,\cdot\,)\,d\bar \mu_{B_n}(x)= \lim_{n\to\infty} \P_B^{\nu_{B'}^n},
$$
where $ \nu_{B'}^n= \int  \P_B^{x} ( V(\cH_{B'})\in\,\cdot\,)\, d\bar \mu_{B_n}(x)$. 
Hence, in particular, $\nu_{B'}^n$  converges in total variation norm to
$$
\nu_{B'}= \bar \Xi (V(\cH_{B'})\in\,\cdot\, )= (\eta\otimes  \bar \Gamma^{\mathrm{deco}})(V(\cH_{B'})\in\,\cdot\,)
$$ and we conclude that
\begin{align}\label{eq894367}
\bar \Xi (\Pi_{B'}\in\,\cdot\, ) = \P_B^{\nu_{B'}}.
\end{align}
At first we consider the measure $\Xi^B$ given by
$$
\frac{d\Xi^B}{d\bar \Xi} =\frac 1{\#\cH_B}
$$
which will later be equal to the branching quasi-process $\Xi$ restricted to $B$.
Using~(\ref{eq894367}) we conclude that  for a measurable set $A$
\begin{align}\begin{split}\label{eq87356}
\Xi^B(\Pi_{B'}\in A)  &= \int \int \frac 1{\#\cH_B} \1_A \, d\P_B^{\mathbf x}\, d\nu_{B'}(\mathbf x)=\int \frac 1{ h(\mathbf x)}  \P^{\mathbf x}(A\cap \mathfrak F_B)\, d\nu_{B'}(\mathbf x)\\
&=\int    \P^{\mathbf x}(A\cap \mathfrak F_B)\, d\mu_{B'}(\mathbf x),
\end{split}\end{align}
where $\mu_{B'}$ is 
 given by
$$
\frac{d\mu_{B'}}{d\nu_{B'}}(\mathbf x) =  \frac 1{ h(\mathbf x)} .
$$
Now we define for every finite set $B'\supset B$ a measure $\Xi^{B'}$ on $\mathfrak T$ via
$$
\Xi^{B'}(A)=\int \frac 1{\P^{\cH_{B'}(\tr)}(\mathfrak F_B)} \int \1_A(\tr^{B'}\cup \tr') \, d\P^{\cH_{B'}(\tr)}(\tr')\,      d\Xi^B(\tr),
$$
where 
$$
\tr^{B'}:=\{(p,i,x)\in \tr: V_\tr(\mathrm{pred}_\tr^j(i))\not \in B' \text{ for } j=1,\dots\}.
$$
Then for $B''\supset B'$ finite, we get with~(\ref{eq87356}) that
\begin{align}\begin{split}\label{eq735672}
\Xi^{B'}(\Pi_{B''}\in A)&= \int \frac 1{\P^{\cH_{B'}(\Pi_{B''}(\tr))}(\mathfrak F_B)}  \int \1_A(\Pi_{B''}(\tr^{B'}\cup \tr')) \, d\P^{\cH_{B'}(\Pi_{B''}(\tr))}(\tr')\,      d\Xi^B(\tr)\\
&=  \int  \int  \frac {\1_{\mathfrak F_B}} {\P^{\cH_{B'}(\tr)}(\mathfrak F_B)}\int \1_A(\tr^{B'}\cup \tr') \, d\P^{\cH_{B'}(\tr)}(\tr')\,      d\P^{\mathbf x} \, d\mu_{B''}(\mathbf x).
\end{split}\end{align}
Now let $\mathcal G_{B'}$ denote the $\sigma$-algebra on $\mathfrak F$ generated  by the mapping $\mathfrak F\ni \mathfrak t\mapsto \mathfrak t^{B'}\in \mathfrak F$ (meaning that one knows all individuals that have no strict predecessor  in $B'$). Using the Markov property  of the branching Markov chain we get that
$$
\E^x[\1_{\mathfrak F_B}|\mathcal G_{B'}]= \IP^{\cH_{B'}}(\mathfrak F_B). 
$$
Note that in the second integral in the second line of~(\ref{eq735672}) all other terms are $\mathcal G_{B'}$-measurable so that we have equality
\begin{align*}
\Xi^{B'}(\Pi_{B''}\in A)&=\int  \int  \1_{\mathfrak F_{B'}}(\tr) \int  \1_A(\tr^{B'}\cup \tr') \, d\P^{\cH_{B'}(\tr)}(\tr')\,      d\P^{\mathbf x}(\tr)\, d\mu_{B''}(\mathbf x)\\
&= \int  \int  \1_{\mathfrak F_{B'}}\,  \1_A \,       d\P^{\mathbf x}\, d\mu_{B''}(\mathbf x).
\end{align*}
Fix an increasing sequence $(B_n)_{n\in\N}$ of finite subsets of $S$ with $B\subset B_1$ and $B_n\uparrow S$.
 The previous display formula  implies that for $1\le n\le m$
$$
\Xi^{B_m}|_{\mathfrak F_{B_n}}=\Xi^{B_n}.
$$
We define $\Xi$ as the monotone limit $\Xi=\lim_{n\to\infty} \Xi^{B_n}$ and note that for $B'\subset S$ finite, we get that
$$
\Xi|_{\mathfrak F_{B'}} (\Pi_{B'}\in\,\cdot\,) = \Xi^{B'} (\Pi_{B'}\in\,\cdot\,) = \IP^{\mu_{B'}}
$$
so that $\Xi$ is a branching quasi-process. In particular, $\Xi^B=\Xi|_{\mathfrak F_B}$ and recalling the definition of $\Xi^B$ we conclude  that $\Xi_B$ is the $B$-biased version of $\Xi$.
\end{proof}

Next, we relate the occupation measure of the $P_h$-quasi-process $\eta$ to the occupation measure of the related branching quasi-process $\Xi$.

\begin{theorem}\label{thm:35676} Let $B$ be a norming region, let  $\eta$ be a $P_h$-quasi-process and denote by~$\mu$ its occupation measure given by
\begin{align}\label{eq_rel}
\mu(x)= \int \sum_{k\in\Z} \1_{\{x\}} (w_k) \,d\eta([w]).
\end{align} Then the in Theorem~\ref{theo:alt_main} associated branching quasi-process $\Xi$  has occupation measure~$\nu$ given by
$$
\nu(x)= \1_{B^c}(x) \,\frac 1{h(x)}\, \mu (x)+\sum_{z\in B} \mu(z) g(z,x).
$$
\end{theorem}

\begin{proof}For a forest $\mathfrak t$ we write $\mathfrak t^B$ for the subforest
$$
\mathfrak t^B=\{(p,i,x)\in \mathfrak t: V_{\mathfrak t}(\pred_{\mathfrak t}^j(i))\not\in B\text{ for }j=0,\dots\}.
$$
Note that the notion is slightly different to the one used in the proof of Theorem~\ref{theo:alt_main}.
We have for $x\in S$
\begin{align*}
\nu(x)&= \int \sum_{i\in \cP_{\mathfrak t}} \1_{\{x\}} (V_{\mathfrak t}(i))\, d\Xi(\mathfrak t)\\
&=\underbrace{\int \sum_{i\in \cP_{\mathfrak t^B}} \1_{\{x\}} (V_{\mathfrak t}(i))\, d\Xi(\mathfrak t)}_{=I}+\underbrace{\int \sum_{i\in \cP_{\Pi_B(\mathfrak t)}} \1_{\{x\}} (V_{\mathfrak t}(i))\, d\Xi(\mathfrak t)}_{=:II}.
\end{align*}
Obviously, expression I vanishes if $x$ is in $B$. Suppose now that $x$ is in $B^c$. Then using the Markov property for the branching quasi-process yields that
\begin{align*}
I&= \frac 1{h(x)} \int \sum_{i\in \cP_{\mathfrak t^B}} \1_{\{x\}} (V_{\mathfrak t}(i))\, \#\cH_B(\Pi_i(\mathfrak t))\, d\Xi(\mathfrak t)\\
&= \frac 1{h(x)}\int \sum_{i\in \cP_{\mathfrak t^B}} \1_{\{x\}} (V_{\mathfrak t}(i))\, \frac{\#\cH_B(\Pi_i(\mathfrak t))}{\#\cH_B}\, d\Xi_B(\mathfrak t)\\
&=\frac 1{h(x)} \int \sum_{i\text{ is a blue} \atop \text{indiv.\ of } \cP_{\bar{\mathfrak t}}} \1_{\{x\}} (V_{\mathfrak t}(i)) \, d\Xi_B\otimes \bar \Gamma^\mathrm{col}(\bar{\mathfrak t}),
\end{align*}
where we used that a vertex $i$ is coloured blue with conditional probability $\frac{\#\cH_B(\Pi_i(\mathfrak t))}{\#\cH_B(\mathfrak t)}$. Together with~(\ref{theo:alt_main}) we get that
$$
I=\frac 1{h(x)} \int  \sum_k \1_{\{x\}}(w_k) \,d\eta([w])=\frac 1{h(x)}\mu(x).
$$
Next, consider II for general $x\in S$. Again we use Theorem~(\ref{theo:alt_main}) and conclude that
\begin{align*}
II=\int  \sum_{i\in \cP_{\Pi_B}} \1_{\{x\}} (V (i))\frac 1{\#\cH_B} \,d \Xi_B=\int  \sum_{i\in \cP_{\Pi_{\mathfrak s }}} \1_{\{x\}} (V(i)) \,d \Xi_B \otimes \bar \Gamma^\mathrm{col}, 
\end{align*}
where $\mathfrak s(\bar{\mathfrak t})$ denotes the unique blue vertex in the coloured graph $\bar{\mathfrak t}$ that lies in $B$. Hence, using that $\Xi_B\otimes \bar \Gamma^\mathrm{col} = \eta\otimes \bar \Gamma^\mathrm{deco}$, that the decoration of an individual in $B$ just returns the branching Markov chain started in a single individual and that $\eta\circ H_B^{-1}=\mu|_B$ we get that
$$
II= \int  \sum_{i\in \cP_{\mathfrak t}} \1_{\{x\}} (V_{\mathfrak t}(i)) \, \Gamma^{\mathrm{deco}}(H_B, d\mathfrak t) \,d\eta= \sum_{z\in B} \mu(z) g(z,x).
$$
\end{proof}

\section{Potential theory for general non-negative matrices}\label{sec:potential_theory}
 In this section we develop potential theory for non-negative matrices $Q:S\times S\to[0,\infty)$ with finite Green's function
 $$
 G=\sum_{n=0}^\infty Q^n.
 $$
 We will work with  one-sided Kuznetsov measures. These are measures on 
$$
\cW=\{(x_n)_{n\in -\N_0}\in (S\cup\{\partial\})^{-\N_0}: \exists \alpha \in \N_0\cup\{-\infty\}\text{ with } x_n\in S \Leftrightarrow \alpha\le n\}
$$
endowed with the product $\sigma$-field.
Moreover, we associate with $Q$ for each $n\in\N_0$ and $x\in S$ the measure $\cQ^x_n$ on $S^{n+1}$ given by
$$
\cQ_n^x(\{(x_0,\dots,x_n)\})=\1_{\{x=x_0\}} q_{x_0,x_1} \ldots q_{x_{n-1},x_n}\qquad (x_0,\dots,x_n\in S).
$$

\begin{definition}Let $(B_n)$ be a sequence of finite subsets of $S$ with $B_n\uparrow S$.
\begin{enumerate}\item
A locally finite measure $\rho$ on $S$ is called \emph{$Q$-excessive} (or \emph{$Q$-subharmonic}),  if for every $x\in S$,
$$
 \rho Q(x)\le\rho(x).
$$
\item Finite measures $(\bar\mu_{B_n}:n\in\N)$ on $S$ are called \emph{consistent entrance family for $Q$} along $(B_n)$ if for every $n\in\N$ and $y\in S$
$$
\bar \mu_{B_n}(y)=\1_{B_n}(y)  \sum_{x\in S} \bar\mu_{B_{n+1}}(x) \sum_{m\in\N_0}\cQ_m^x(X_m=y, X_0,\dots,X_{m-1}\not \in B_n).
$$
\item For an excessive measure $\nu$ we call $\nu^\mathrm{pot}=\nu(\mathbb I-Q)$ the \emph{purely excessive} or \emph{potential measure} and $\nu^\mathrm{inv}= \nu -\nu^\mathrm{pot} G$ the \emph{invariant measure} associated with $\nu$.
\item A measure $\cK^\nu$ on $\cW$ is called (one-sided)  \emph{Kuznetsov measure} associated with an excessive measure $\nu$, if for every $n\in\N$ and $x_{-n},\dots,x_0\in S$
$$
\cK^\nu(\{(\ldots,\partial,x_{-n},\dots,x_0)\})= \nu^\mathrm{pot}_{x_{-n}} q_{x_{-n},x_{-n+1}}\ldots q_{x_{-1},x_0}
$$
and
$$
\cK^\nu\Bigl(\prod_{m\in \Z\cap(-\infty, -n-1]} (S\cup \{\partial\}) \times \{(x_{-n},\dots,x_0)\}\Bigr)=\nu_{x_{-n}} q_{x_{-n},x_{-n+1}}\ldots q_{x_{-1},x_0}.
$$
\item The matrix $(\hat p_{x,y})_{x,y\in S}$ given by
$$
\hat p_{x,y}= \frac 1{\nu(x)} \nu(y) q_{y,x} 
$$
is substochastic and the respective Markov family $(\hat P^x: x\in S)$ is called $\nu$-adjoint Markov chain.
\end{enumerate}
\end{definition}

\begin{remark}\begin{enumerate}\item
If $Q$ is the intensity operator for a branching Markov chain (as introduced in~(\ref{intensity_operator})) we can rephrase the condition of a consistent entrance family as follows: a family  $(\bar \mu_{B_n})_{n\in\N}$ of finite measures on $S$ is a consistent entrance family if  
$$
\bar \mu_{B_n}(y)= \int \E^x[\#(\cH_{B_n}\cap([0,1]\times \{y\})] \,d \bar \mu_{B_{n+1}}(x).
$$
\item The one-sided Kuznetsov measure plays a similar role as the Kuznetsov measure in the classical theory. However we should stress that even in the case where~$Q$ is substochastic, the definition does not coincide with the original one. 
\end{enumerate}
\end{remark}

\begin{proposition}Let $\nu$ be an excessive measure.
\begin{enumerate}\item
The associated one-sided Kuznetsov measure $\cK^\nu$ is the image measure of $\hat P^{\nu}$ under time-reversal.
\item The invariant measure $\nu^\mathrm{inv}$ is $Q$-invariant and we have
$$
\nu= \nu^\mathrm{inv}+\nu^\mathrm{pot} G.
$$
\end{enumerate}
\end{proposition}

\begin{proof}We  verify 1. For $n\in\N_0$ and $x_{-n},\dots,x_0\in S$ one has
\begin{align*}
\hat P^\nu (X_0=x_0,\dots, X_n=x_{-n})= \nu_{x_0} \hat p_{x_0,x_{-1}}\ldots \hat p_{x_{-n+1},x_{-n}}= \nu _{x_{-n}} q_{x_{-n},x_{-n+1}}\ldots q_{x_{-1},x_0}
\end{align*}
and
\begin{align*}
\hat P^\nu (X_0=x_0,\dots, X_n=x_{-n},X_{n+1}=\partial)&=  \nu_{x_0} \hat p_{x_0,x_{-1}}\ldots \hat p_{x_{-n+1},x_{-n}} \lambda_{x_{-n}}\\
&=\lambda_{x_{-n}}\nu _{x_{-n}} q_{x_{-n},x_{-n+1}}\ldots q_{x_{-1},x_0}\\
&= \nu^\mathrm{pot}_{x_{-n}} q_{x_{-n},x_{-n+1}}\ldots q_{x_{-1},x_0},
\end{align*}
where $\lambda_x=1-\sum_{y\in S} \hat p_{x,y}=1- \frac 1{\nu_x} \sum_{y\in S} \nu_y \,p_{y,x}$ is the probability that the $\hat P$-chain dies in the state $x\in S$. 
The verification of 2.\ is straight-forward.
\end{proof}

Every one-sided  Kuznetsov measure satisfies a Markov-property, see also \cite[p.~163]{Kem76}.

\begin{lemma}[Markov property] \label{le:MP} Let $\nu$ be a $Q$-excessive measure. For a measurable set $A$ in $\cW$, $n\in\N$, $B\subset S^{n+1}$ and $x\in S$ one has
$$
\cK^\nu((X_{m-n})_{m\in-\N_0}\in A, X_{-n}=x, (X_{-n},\dots X_0)\in B)= \cK^\nu(A \cap\{X_{0}=x\})\, \cQ^x_n(B).
$$
\end{lemma}

\begin{proof}
First  consider the sets  $A=\{X_{-k}=x_{-k},\dots,X_0=x_0\}\subset \cW$ and $B=\{(y_0,\dots,y_m)\}$ with $k\in\N$ and $x_{-k},\dots,x_0,y_0,\dots,y_n\in S$. For these sets one immediately gets that
\begin{align*}
\cK^\nu((X_{m-n}&)_{m\in-\N_0}\in A, X_{-n}=x, (X_{-n},\dots X_0)\in B)\\
&= \nu_{x_{-k}} q_{x_{-k},x_{-k+1}} \ldots q_{x_{-1},x_0} \1_{\{x_0=x=y_0\}} q_{y_0,y_1}\ldots q_{y_{n-1},y_n}\\
&= \cK^\nu( A\cap\{X_{0}=x\}) \,\cQ^x_n(B).
\end{align*}
In complete analogy one obtains the formula in the case where $x_{-k}=\partial$ instead of $x_{-k}\in S$. The sets form a $\cap$-stable generator and it thus follow the result by standard arguments.
\end{proof}
For a finite subset $B'\subset S$ we set
$$
Q^{B'}:B'\times B'\to [0,\infty),  (x,y)\mapsto \sum_{n=1}^\infty \cQ_n^x(T_{B'}^*=n,X_n=y),
$$
where $T_{B'}^*$ denotes the first return time to $B'$.

\begin{theorem}\label{thm:main1}
There are isomorphisms between the cones of all
\begin{enumerate}
\item $Q$-excessive measures $\nu$,
\item $P_h$-excessive measures $\mu$,
\item $P_h$-quasi-processes $\eta$ and
\item the consistent entrance families along $(B_n)$, $(\bar \mu_{B_n}:n\in\N)$,
\end{enumerate}
that satisfy
\begin{enumerate}
\item[(i)] $\eta=\cK^\nu|_{\{T_B=0\}}\circ [\,\cdot\,]^{-1}$, where $[\cdot]:D\to D^*$ is the canonical embedding,
\item[(ii)] $\displaystyle{ \nu(x)=\1_{B^c} (x) \frac{\mu(x)}{h(x)}+ \sum_{z\in B} \mu(z) g(z,x)}$ and
\item[(iii)] $\mu$ is the occupation measure of $\eta$, i.e., $\displaystyle{ \mu(x)=\int \sum_{n\in\Z} \1_{\{x\}}(w_n)\, d\eta([(w_n)])}$,
\item [(iv)]  $\displaystyle{\bar \mu_{B_n}(x)= \cK^\nu|_{\{T_{B_n}=0\}}(X_0=x)}$ \ and, for $n\in\N$ with $B\subset B_n$, \ $\displaystyle{\bar \mu_{B_n}(x)= \frac 1{h(x)} \eta(H_{B_n}=x)}$,
\item[(v)] for $n\in\N$, $\nu|_{B_n}=\bar \mu_{B_n} G|_{B_n\times B_n}$ and  $\bar \mu_{B_n}=\nu|_{B_n}(\II-Q^{B_n})$.
\end{enumerate}
\end{theorem}

We first verify the relation (ii) in the case where $\nu=g(x,\cdot)$ for a $x\in S$. 
\begin{proposition}\label{prop8357}
Let $x\in S$ and $\nu$ and $\mu$ be the measures given by
$$
\nu(y)= g(x,y) \text{ \ and \ } \mu(y)= h(x) E^x_h\Bigl[\sum_{n=0}^\infty\1_{\{y\}}(X_n)\Bigr]\qquad(y\in S).
$$
Then for all $y\in S$
$$
\nu(y)=\1_{B^c}(y)\frac{\mu(y)}{h(y)} +\sum_{z\in B} \mu(z) g(z,y).
$$
\end{proposition}

\begin{proof}
The statement is trivial if $x\in B$ since then $\mu=\1_{\{x\}}$. We restrict attention to the case where $x\in B^c$.
For $y\in B$, one has 
$$
\mu(y)= h(x)\, P^x_h(H_B=y) =\sum_{n=1}^\infty h(x)\, P^x_h(T_B=n, X_n=y) =\sum_{n=0}^\infty ((Q|_{B^c\times B^c})^n Q|_{B^c\times B})_{x,y}
$$
and, for $y\in B^c$,
\begin{align*}
\mu(y)=h(x)\, E_h^x\Bigl[\sum_{n=0}^\infty \1_{\{y\}}(X_n)\Bigr]  = h(y) \sum_{n=0}^\infty (Q|_{B^c\times B^c})^n_{x,y}.
\end{align*}
Consequently, for $y\in B$,
\begin{align*}
\nu(y)&= \sum_{n=0}^\infty\sum_{m=0}^\infty \bigl((Q|_{B^c\times B^c})^nQ|_{B^c\times B} Q^m|_{B\times S} \bigr)_{x,y}  = \sum_{z\in B} \mu(z)\, g(z,y)
\end{align*}
and, for $y\in B^c$,
\begin{align*}
\nu(y)&=\sum_{n=0}^\infty \bigl((Q|_{B^c\times B^c})^n\bigr)_{x,y}+ \sum_{n=0}^\infty\sum_{m=0}^\infty \bigl((Q|_{B^c\times B^c})^nQ|_{B^c\times B} Q^m|_{B\times S} \bigr)_{x,y} \\
&=\frac{\mu(y)}{h(y)}+ \sum_{z\in B} \mu(z)\, g(z,y).
\end{align*}
\end{proof}

\begin{proposition}\label{prop435876}
The  restricted operator $G|_{B\times B}$ is the Green's function of
$$
Q^B:B\times B\to [0,\infty), (x,y)\mapsto \sum_{n=1}^\infty \cQ_n^x(T^*_B=n, X_n=y),
$$
where $T^*_B$ is the first return time to $B$.
\end{proposition}

\begin{proof}
The proof is standard and just given for convenience. By induction one verifies that for $x,y\in B$ 
$$
(Q^B)^k_{x,y}= \sum_{\ell=k}^\infty \cQ_\ell^x(X_\ell=y, X_1,\dots,X_{\ell}\text{ has exactly $k$ entries in $B$}).
$$
Consequently,
\begin{align*}
\sum_{k=0}^\infty (Q^B)^k_{x,y}&=\sum_{k=0}^\infty \sum_{\ell=k} \cQ_\ell^x(X_\ell=y, X_1,\dots,X_{\ell}\text{ has exactly $k$ entries in $B$})\\
&=\sum_{\ell=0}^\infty \underbrace{\sum_{k=0}^\ell  \cQ_\ell^x(X_\ell=y, X_1,\dots,X_{\ell}\text{ has exactly $k$ entries in $B$})}_{=\cQ_\ell^x(X_\ell=y)}=g(x,y).
\end{align*}
\end{proof}

\begin{proof}We prove the following claims:
\begin{enumerate}\item For every $Q$-excessive measure $\nu$, 
   $\cK^\nu|_{\{T_B=0\}}\circ [\,\cdot\,]^{-1}$ is a $P_h$-quasi-process.
 \item  $\cK^\nu|_{\{T_B=0\}}\circ [\,\cdot\,]^{-1}$ has occupation measure $\mu$ satisfying (ii).
\item  Given a $Q$-excessive measure $\nu$, there exists a unique measure $\mu$ satisfying (ii).
\item  Given a $P_h$-excessive measure $\mu$, the measure  $\nu$ given by (ii) is $Q$-excessive.
\item For every $Q$-excessive measure $\nu$ and increasing sequence $(B_n)$ with $B_n\uparrow S$, the measures $(\bar \mu_{B_n}:n\in\N)$ given by
$$
\bar \mu_{B_n}(x)= \cK^\nu|_{\{T_{B_n}=0\}}(X_0=x)
$$
form a consistent entrance family and one has 
\begin{enumerate} \item[(a)] for every finite $B'\supset B$, $\cK^\nu|_{\{T_B=0\}} (H_{B'}=x)=h(x)  \, \bar \mu_{B'}(x) $ and
 \item[(b)] for every $n\in\N$, $\nu|_{B_n}=\bar \mu_{B_n} G|_{B_n\times B_n}$ and $\nu|_{B_n}(\II-Q^{B_n})=\bar \mu_{B_n}$.
\end{enumerate}
\item For every consistent entrance family $(\bar\mu_{B_n}:n\in\N)$  there is a unique $Q$-excessive measure $\nu$  satisfying
$$
\bar \mu_{B_n}G|_{B_n\times B_n}= \nu|_{B_n}\text{, \ for every $n\in\N$}.
$$
\end{enumerate}
We explain how the statement of the theorem is then obtained. To a $Q$-excessive measure $\nu$ we associate the $P_h$-quasi-process $\eta=\cK^\nu|_{\{T_B=0\}}\circ [\,\cdot\,]^{-1}$ (Claim 1) and the respective occupation measure $\mu$ (Claim 2). As is well known there is an isomorphism mapping all $P_h$-quasi-processes to the cone of $P_h$-excessive measures that takes a quasi-process to its occupation measure, see for instance \cite{Get90}. So to show that every $P_h$-quasi-process, resp., $P_h$-excessive measure is reached by the above construction we use (Claim~4) to choose~$\nu$. Then Claim~3 implies that the measure $\mu$ satisfying (ii) is unique so that  $\mu$ is indeed the occupation measure of the respective $P_h$-quasi-process associated to $\nu$. This completely establishes the isomorphism between the cones of $Q$-excessive measures, $P_h$-excessive measures and $P_h$-quasi-processes satisfying properties (i)-(iii).

It remains to consider the related consistent entrance families.  Claim~5 relates a $Q$-excessive measure $\nu$ to a consistent entrance family $(\bar \mu_{B_n})_{n\in\N}$ satisfying (iv) and (v).  Conversely,  Claim~6 implies existence of a $Q$-excessive measure satisfying~(v) for every consistent entrance family.\smallskip

 \noindent\textbf{Proof of Claim 1:} We  verify that $\eta:=\cK^\nu|_{\{T_B=0\}}\circ [\,\cdot\,]^{-1}$
is a $P_h$-quasi-process. 
One has for $B'\supset B$ and for $n\in\N_0$ and  $x_0,\dots,x_n\in S$:
$$\eta\bigl(\pi_{B'}= (x_0,\dots,x_n,\partial,\dots )\bigr)= \cK^\nu|_{\{T_B=0\}}\bigl(T_{B'}=-n, T_B=0, X_{-n}=x_0,\dots, X_0=x_n\bigr).
$$
Note that the latter measure is zero in each of the following cases:
(a) $x_n\in B^c$, (b)  one of the states $x_0,\dots,x_{n-1}\in B$ or (c)  $x_0\not \in B'$. 
Assume that none of the latter cases is fulfilled. 
We denote by $(\hat P^x:x\in S)$  the $\nu$-adjoint chain and by $\bar T_{\bar B}= \sup\{n\in\N_0: X_n\in \bar B\}$ ($\bar B\subset S$) the last exit time from $B$ (with the convention that $\sup\emptyset =-\infty$). We get that
\begin{align*}
\eta\bigl(\pi_{B'}= (x_0,\dots,x_n,\partial,\dots )\bigr)&=  \hat P^{\nu} \bigl(\bar T_B=0, X_0=x_n,\dots, X_n=x_0, \bar T_{B'}=n\bigr)\\
&=\hat P^{\nu} \bigl(X_0=x_n,\dots, X_n=x_0) \,\hat P^{x_0}(\bar T_{B'}=0).
\end{align*}
Using that $h(x_n)=1$ we get that
\begin{align*}\hat P^{\nu} \bigl(X_0=x_n,\dots, X_n=x_0)&=\nu(x_n) \hat p _{x_n,x_{n-1}}\ldots \hat p _{x_1,x_{0}}\\
&=\nu(x_0) \, q_{x_0,x_1} \ldots q_{x_{n-1},x_n} h(x_n) \\
&= \nu(x_0) \,h(x_0)\, p^h_{x_0,x_1} \ldots p^h_{x_{n-1},x_n}\end{align*}
and $\nu(x_0) \hat P^{x_0}(\bar T_{B'}=0)= \cK^\nu(X_0=x_0, T_{B'}=0)=:\bar \mu_{B'}(x_0)$. Consequently,
\begin{align}\label{eq98571}
\eta\bigl(\pi_{B'}= (x_0,\dots,x_n,\partial,\dots )\bigr)=\bar \mu_{B'}(x_0) \,h(x_0)\, p^h_{x_0,x_1} \ldots p^h_{x_{n-1},x_n}.
\end{align}
This equality also holds in the case where one of the properties (a), (b) or (c) holds in which case the left and right hand side are zero. Consequently, $\eta$ is a $P_h$-quasi-process. 
\smallskip

 \noindent\textbf{Proof of Claim 2:} We compute the occupation measure $\mu$ of $\eta:=\cK^\nu|_{\{T_B=0\}}\circ [\,\cdot\,]^{-1}$. Take $x\in S$ and $B'\supset B\cup\{x\}$. Then with~(\ref{eq98571})
\begin{align*}
\mu(x) = \int \sum_{n} \1_{\{x\}}(X_n)\, d\eta\circ{\pi_{B'}}= E_h^{h\cdot \bar \mu_{B'}}\Bigl[\sum_{n=0}^\infty \1_{\{x\}}(X_n)\Bigr]= \sum_{z\in B'} \bar \mu_{B'}(z)  \,h(z)\,  g_h(z,x),
\end{align*}
where $g_h(z,x)=E^z_h[\sum_{n=0}^\infty\1_{\{y\}}(X_n)]$ is the Green's function for $P_h$. Conversely, using the Markov property (Lemma~\ref{le:MP}) we get that
\begin{align*}
\nu(x)&=\cK^\nu(X_0=x,T_B=0) =\sum_{z\in B'}\sum_{n\in\N_0} \cK^\nu(X_{-n}=z, T_{B'}=-n,X_0=x,T_B=0)\\
&=\sum_{z\in B'} \underbrace{\cK^\nu( X_0=z, T_{B'}=0)}_{=\bar \mu_{B'}(z)} \sum_{n=0}^\infty \underbrace{\cQ_n^z(X_n=x)}_{=(Q^n)_{z,x}}= \sum_{z\in B'} \bar \mu_{B'}(z) \,g(z,x).
\end{align*}
Proposition \ref{prop8357} gives a linear relation between  the Green's function of $Q$ and $P_h$. Together with  the latter two display formulas it entails
that
$$
\nu(x)=\1_{B^c}(x)\frac{\mu(x)}{h(x)} +\sum_{z\in B} \mu(z) g(z,x).
$$
%
%
%
%
\smallskip

 \noindent\textbf{Proof of Claim 3:}
 First consider the equation restricted to the finite set $B$: one has
$$
\nu|_B= \mu|_B G|_{B\times B}
$$
By Proposition~\ref{prop435876}, $G|_{B\times B}$ is the Green's function of $Q^B$ so that
$$
G|_{B\times B}(\II-Q^B)= \II.
$$
Consequently,
$$
\nu|_B(\II-Q^B)= \mu|_B
$$
and $\mu$ is uniquely determined on $B$. Obviously, then (ii) yields also uniqueness of $\mu$ on~$B^c$.
\smallskip

 \noindent\textbf{Proof of Claim 4:} One has
\begin{align*}
\sum_{y\in S} \nu (y) q_{y,x}&=\sum_{y\in B^c}  \frac {\mu(y)}{h(y)}q_{y,x} +\sum_{y\in S} \sum_{z\in B} \mu(z) g(z,y) q_{y,x}\\
&=\frac 1{h(x)} \underbrace{\sum_{y\in B^c} \mu(y)p^h_{y,x}}_{\le \mu(x)}+  \sum_{z\in B} \mu(z) \underbrace{\sum_{y\in S}  g(z,y) q_{y,x}}_{=g(z,x)-\1_{\{z=x\}}}.
\end{align*}
  If $x\in B^c$, then the right hand side is less than or equal to $\nu(x)$. If $x\in B$, then $h(x)=1$ and we get again that the right hand side is less than or equal to $\nu(x)$.
 \smallskip

 \noindent\textbf{Proof of Claim 5:} One has by the Markov property, for $n\in\N$ and $x\in S$, 
\begin{align*}
\bar \mu_{B_n}(x)&=\cK^\nu|_{\{T_{B_n}=0\}}(X_0=x)=\sum_{z\in B_{n+1}} \sum_{k\in\N_0} \cK^\nu(   T_{B_{n+1}}=-k, X_{-k}=z, T_{B_n}=0,X_0=x)\\
&=\1_{B_n}(x) \sum_{z\in B_{n+1}} \underbrace{\cK^\nu(   T_{B_{n+1}}=0, X_{0}=z)}_{=\bar \mu_{B_{n+1}}(z)} \,\sum_{k=0}^\infty \cQ_k^z(X_k=x, X_0,\dots,X_{k-1}\not \in B)
\end{align*}
so that $(\bar\mu_{B_n}:n\in\N)$ is a consistent entrance family along $(B_n)$. We verified (a)  in the proof of the first claim, see~(\ref{eq98571}). The first part of property (b) follows since for general  finite $B'\subset S$ and $x\in B'$
\begin{align*}
\nu(x)&= \cK^\nu(X_0=x)= \sum_{z\in B'}\sum_{m=0}^\infty \cK ^\nu(T_{B'}=-m, X_{-m}=z, X_0=x) \\
&= \sum_{z\in B'}\underbrace{\cK ^\nu(T_{B'}=0, X_{0}=z)}_{=\bar \mu_{B'}(z)} \sum_{m=0}^\infty \underbrace{\cQ_m^z(X_m=x)}_{=(Q^m)_{z,x}}\\
&= \sum_{z\in B'}\bar \mu_{B'}(z) g(z,x).
\end{align*}
By Proposition~\ref{prop435876}, $G|_{B'\times B'}$ is the Green's function of $Q^{B'}$ so that it has inverse $\II-Q^{B'}$ and we obtain the second part of property (b).\smallskip

 \noindent\textbf{Proof of Claim 6:}
We define a measure $\nu$ via $\nu|_{B_n}=\bar \mu_{B_n} G|_{B_n\times B_n}$ for $n\in\N$. To show that $\nu$ is well-defined we verify consistency: for $n\in\N$ and $x\in B_n$, one has 
\begin{align*}
\bar \mu_{B_{n+1}} G|_{B_{n+1}\times B_{n+1}} (x)&= \sum_{z\in B_{n+1}} \bar \mu_{B_{n+1}}(z)\sum_{m=0}^\infty \cQ_m^z(X_m=x)\\
&=\sum_{z\in B_{n+1}} \bar \mu_{B_{n+1}}(z)\sum_{m=0}^\infty  \sum_{z'\in B_n}\sum_{\ell=0}^m  \underbrace{\cQ_m^z(T_{B_n}= \ell, X_\ell=z', X_m=x)}_{=\cQ_{\ell}^z(X_\ell= z', X_0,\dots, X_{\ell-1} \in B_n^c) \,\cQ_{m-\ell}^{z'} (X_{m-\ell} =x)}\\
&= \sum_{z'\in B_n}\underbrace{\sum_{z\in B_{n+1}} \bar \mu_{B_{n+1}}(z)\sum_{\ell=0}^\infty\cQ_{\ell}^z(X_\ell= z', X_0,\dots, X_{\ell-1} \in B_n^c)}_{=\bar \mu_{B_n}(z')} \,\underbrace{\sum_{m'=0}^\infty\cQ_{m'}^{z'} (X_{m'} =x)  }_{=g(z',x)}\\
&=\bar \mu_{B_n} G|_{B_n\times B_n} (x).
\end{align*}
Next, we verify that  $\nu$ is $Q$-excessive:  by monotone convergence, one has, for $x\in S$, 
\begin{align*}
\sum_{y\in S} \nu _y q_{y,x}&= \lim_{n\to\infty} \sum_{y\in B_n} \nu _y q_{y,x}=\lim_{n\to\infty} \bar \mu_{B_n} G|_{B_n\times B_n}  Q|_{B_n\times B_n} (x)\\
&\le\lim_{n\to\infty} \bar \mu_{B_n} G|_{B_n\times B_n}(x)= \nu(x).
\end{align*}
\end{proof}

An immediate consequence of Theorems~\ref{thm:35676} and~\ref{thm:main1} is the following theorem.

\begin{theorem}\label{theo:57}
Let $B$ be a norming region and $\nu$ be a $Q$-excessive measure.\begin{enumerate}
\item  There exists a unique $P_h$-quasi-process $\eta$  such that the occupation measure $\mu$ of~$\eta$ satisfies~(\ref{eq_rel}). 
\item If $\eta$ is decorable, then Theorem~\ref{theo:alt_main} relates $\eta$ to a branching quasi-process $\Xi$ with occupation measure $\nu$. 
\item Every branching quasi-process $\Xi$ is obtained by the above two steps when starting with its occupation measure $\nu$.
\end{enumerate}
\end{theorem}

\section{A spine construction for branching quasi-processes}\label{sec6}

We consider the particular case where the intensity operator $Q$ is sub-Markovian. In that case, we can work with $Q$-Markov families.
The central technical result is the following proposition.

\begin{proposition}\label{prop456} Let $Q$ be  sub-Markovian, i.e.,  $m_x\le 1$ for all $x\in S$. We relate  a $Q$-excessive measure $\nu$ with the  unique  $Q$-quasi-process $\bar \eta$ satisfying
$$
\nu(A)=\int \sum_{n\in\Z} \1_A(w_n)\, d\bar \eta([(w_n)]),\text{ \ for }A\subset S.
$$ 
Theorem~\ref{thm:main1} associates~$\nu$ with the $P_h$-quasi-process $\eta$ given by
$$
\eta= \bar \eta|_{D_B^*}\circ \mathrm{death}_B^{-1}, 
$$
where $\mathrm{death}_B: D \to D$ takes the path $w=(w_n)$ to the path
$$
n\mapsto \begin{cases} w_n, &\text{ \ if } n\le T_B(w),\\ \partial, &\text{ \ if } n>T_B(w).\end{cases}
$$
\end{proposition}

\begin{remark}[Spine construction of branching interlacements]\label{rem:3578}
The theorem entails an algorithm for the construction of  interlacements for branching Markov chains in the case where $Q$ is sub-Markovian. Suppose we want to generate a branching interlacement for a given  $Q$-excessive measure $\nu$. Starting with the unique  $Q$-quasi-process $\bar \eta$ with occupation measure $\nu$ we provide an algorithm that generates all  trees hitting a set $B$ of a random interlacement with occupation measure $\nu$. 
One proceeds as follows:
\begin{enumerate}\item take an interlacement with intensity measure $\bar \eta$,
\item erase all paths that do not hit $B$,
\item remove all entries of the paths that have strict predecessors in $B$ (apply the $\mathrm{death}_B$ operation),
\item then apply independently the decoration procedure on the  remaining paths and
\item keep every tree $\mathfrak t$ (generated in 4) independently with probability $1/\cH_B(\mathfrak t)$.
\end{enumerate}
Indeed, this generates a branching interlacement with occupation measure $\nu$ (provided that there exists one): By Proposition~\ref{prop456}, operations one to three yield an $P_h$-interlacement with intensity measure $\eta$. After  step 4 the generated trees form a Poisson point process with intensity measure $\Xi_B$ where $\Xi_B$ is the $B$-biased variant of a branching quasi-process $\Xi$ with occupation measure $\nu$. Recall that we have on  $\mathfrak T_B$ that  $\frac {d\Xi}{d\Xi_B}=\frac 1{\#\cH_B}$, so that the last step produces a Poisson point process with intensity measure $\Xi|_{\mathfrak T_B}$.
%
\end{remark}

\begin{proof}[Proof of Proposition~\ref{prop456}]
First we show that $\eta:=\bar \eta|_{D_B^*}\circ \mathrm{death}_B^{-1}$ is a $P_h$-quasi-process.
Let $B'\supset B$ be a finite set and denote by $\bar \mu_{B'}$ the $B'$ entrance measure of $\bar   \eta$, i.e., for $A\subset S$
$$
\bar \mu_{B'}(A)= \bar \eta|_{D_{B'}^*} ( H_{B'}\in A).
$$
Obviously,  $\eta$-almost every path enters $B$. Let $n\in\N_0$, $w_0,\dots,w_{n-1}\in S\backslash B$ and $w_n\in B$, and set $w_{k}=\partial$ for $k\ge n+1$. Then
\begin{align*}
 \eta (\pi_{B'}= (w_k)_{k\in\N_0}) &= P^{\bar \mu_{B'}}(X_k=w_k, \text{ for }k=0,\dots,n)=\bar \mu_{B'}(w_0) \prod _{k=1}^n q_{w_{k-1},w_k}\\
 &=  \bar \mu_{B'}(w_0) \, h(w_0) \,  \prod _{m=1}^n q^h_{w_{m-1},w_m} = \bar \mu_{B'}(w_0) \, h(w_0) \,P_h^{w_0}((w_k)_{k\in\N_0}).
\end{align*}
All other paths have measure zero and we showed that $\eta$ is a $P_h$ quasi process with $B'$ entrance measure $\mu_{B'}$ satisfying
$$
\frac {d \mu_{B'}}{d \bar \mu_{B'}}=h.
$$

Next, we verify that  the occupation measures $\mu$ of $\eta$ and $\nu$ of $\bar \eta$ satisfy  for all $y\in S$
$$\nu(y)= \1_{B^c} (y) \, h(y)^{-1}\, \mu(y) +\sum_{x\in B} \mu(x)\, g(x,y).
$$
For $y\in B'$ we get with  the quasi-process-property that
$$
\nu(y)=\int \sum_{n\in\Z} \1_{\{y\}} (w_n)\, d\eta([w])= E^{\bar \mu_{B'}}\Bigl[\sum_{n\in\N_0} \1_{\{y\}}(X_n)\Bigr] .
$$
We observe that
\begin{align}\label{eq746599}
P^{\bar \mu_{B'}}(X_n=y)=P^{\bar \mu_{B'}}(X_n=y, T_B<n)+P^{\bar \mu_{B'}}(X_n=y, T_B\ge n) 
\end{align}
and note that the contribution of the first summand on the right hand side satisfies
$$
\sum_{n\in\N} P^{\bar \mu_{B'}} (X_n=y, T_B< n)= \sum_{z\in B} \underbrace{\bar \mu_B(z)}_{=\mu_B(z)} (g(z,y)-\1_{\{z=y\}}).
$$
Moreover, using the  Markov property we get that for $z\in S$
 \begin{align*}
P^z(X_n=y, T_B\ge n)&= \frac 1{h(y)} P^z(X_n=y, T_B\in[ n,\infty))= \frac {h(z)}{h(y)} P_h^z(X_n=y),
\end{align*}
so that the second summand of (\ref{eq746599}) satisfies
\begin{align*}
\sum_{n\in\N_0} P^{\bar \mu_{B'}}(X_n=y, T_B\ge n) &=\sum_{n\in\N_0}\sum_{z\in B'} \bar \mu_{B'}(z)\frac {h(z)}{h(y)} P_h^z(X_n=y)\\
&=\frac {1}{h(y)}\sum_{n\in\N_0}\sum_{z\in B'}  \mu_{B'}(z) P_h^z(X_n=y)=\frac 1{h(y)} \mu(y).
\end{align*}
Altogether, we conclude with the property that  $\mu(x)=\mu_B(x)$ for $x\in B$ that
\begin{align*}
\nu(y)=\sum_{n\in\N_0} P^{\bar \mu_{B'}}(X_n=y)=\sum_{x\in B}  \mu(x) (g(x,y)-\1_{\{x=y\}})+\frac 1{h(y)} \mu(y).
\end{align*}
For $y\in B$, one has  $h(y)=1$ and, hence, $\nu(y)=\sum_{x\in B}  \bar \mu_B(x) g(x,y)$, and for $y\in B^c$, one has
$$
\nu(y)=\frac {\bar \mu_B(y)}{h(y)}+ \sum_{x\in B}  \mu(x) g(x,y).
$$
\end{proof}

\section{Decorability}\label{sec7}

In the following, $B$ is again a norming region and
$$
h(x)= \E^x[\#\cH_{B}]
$$
for $x\in S$.

\begin{theorem}\label{thm:485} Let $B\subset S$ be a norming region 
and suppose that
$$
\bar m_x:=\sum_{k\in\N} (k-1)  \, d_x^\mathrm{s.b.}(k)
$$
is finite\footnote{This assumption is equivalent to the existence of the second moment of each offspring distribution $d_x$.} for all $x\in S$. 
If
\begin{align}\label{cond:deco}
C:=\sup_{x\in S} \frac 1{h(x)} \sum_{y\in B^c} \frac {\bar m_y}{m_y}  \, g(x,y) \,h(y)^2  <\infty,
\end{align}
then every $P_{h}$-quasi-process $\eta$  is decorable. Moreover, for every $x\in S$,
$$
\frac 1{4C+2} h(x) \le \P^x(\mathfrak T_B)\le h(x).
$$
\end{theorem}

\begin{proof}
We verify the lower bound for the probability $\P^x(\mathfrak T_B)$.
Note that 
\begin{align}\label{eq234568}
\P^x(\mathfrak T_B)&=h(x)\int \frac {1}{\#\cH_B} \,d\P_B^x= h(x)\int \int \frac {1}{\#\cH_B} d\Gamma^\mathrm{deco}(X,\,\cdot\,)\, dP^x_h.
\end{align}
By Proposition~\ref{prop8357}, the process $P_h^x$ has occupation measure $\mu^{(x)}$ satisfying
$$
g(x,y)=\1_{B^c}(y) \frac{h(x)}{h(y)}\mu^{(x)}(y)+ h(x) \sum_{z\in B}  \mu^{(x)}(z)\,g(z,y).
$$
Hence, for $y\in B^c$, one has
$$
\mu^{(x)}(y)\le \frac {h(y)}{h(x)} g(x,y).
$$
Since $\sum_{y\in B} \mu^{(x)}(y)=1$ and the decoration of elements in $B$ does not add additional individuals we get that
\begin{align*}
\int \int (\#\cH_B-1)\, d\Gamma^\mathrm{deco}(X,\,\cdot\,)\, dP^x_h& =\sum_{y\in B^c}\mu^{(x)}(y)\underbrace {\int \#\cH_B \,d \Gamma^{\mathrm{deco}}(y,\,\cdot\,)}_{=\frac{\bar m_y}{m_y} h(y)}\\
&\le \sum_{y\in B^c} \frac{h(y)}{h(x)} g(x,y) \frac{\bar m_y}{m_y} h(y)\le C.
\end{align*}
Hence, the Markov inequality implies that
$$
\int\int \1\{(\#\cH_B-1)\ge 2C\} \, d\Gamma^\mathrm{deco}(X,\,\cdot\,)\, dP^x_h\le \frac 12
$$
so that we get with~(\ref{eq234568}) that
$$
\P^x(\mathfrak T_B)\ge \frac 1{4C+2}h(x).
$$

 Let $\eta$ be a $P_{h}$-quasi-process.
 We need to show decorability of the quasi paths, i.e., to show that
$$
\eta\otimes \Gamma^{\mathrm{deco}}
$$
attains values in the set of ordered \emph{transient} trees $\cT_B$. 
Let $N_y(\tr)$ denote the numbers of individuals with position in $y\in S$.
By  Theorem~\ref{thm:spine}, we have for $x\in S$
\begin{align*}
\E_B^x [N_y] & = \int  \int N_y \, d \Gamma^{\mathrm{deco}}(w,\,\cdot\,)\,    dP_{ h}^x(w)=\int \sum_k  \int N_y \, d \Gamma^{\mathrm{deco}}(w_k,\,\cdot\,)\,    dP_{ h}^x(w)\\
&=\int    \int N_y \, d \Gamma^{\mathrm{deco}}(z,\,\cdot\,)  \, d\mu^{(x)}(z)\\
&\le \sum_{z\in B^c} g(x,z)\frac {h(z)}{h(x)}   \int N_y \, d \Gamma^{\mathrm{deco}}(z,\,\cdot\,) + \max_{u\in B} g(u,y).
\end{align*}
It remains to bound the latter sum. By definition of the decoration procedure we get for a single $z\in B^c$ that
$$
\int N_y \, d \Gamma^\mathrm{deco}(z,\cdot) =   \frac {\bar m_z}{m_z} g(z,y) +\1_{\{z=y\}}\Bigl(1-\frac{\bar m_y}{m_y}\Bigr).
$$
Note that
$g(z,B)\le h(z) \max_{u\in B} g(u,B)$  and $g(z,B) \ge g(z,y) h(y)$
so that  
\begin{align}\label{eq8956} 
g(z,y) \le \frac 1{h(y)} \max_{u\in B} g(u,B) \,h(z).
\end{align}
Hence,
\begin{align*} \sum_{z\in B^c}g(x,z) \frac {h(z)}{h(x)} & \int N_y \, d \Gamma^{\mathrm{deco}}(z,\,\cdot\,) \le \max_{u\in B}g(u,B) \Bigl( 1+\frac 1{h(y)}  \underbrace{\sum_{z\in B^c} \frac{\bar m_z}{m_z}g(x,z)\frac{h(z)^2}{h(x)}}_{\le C} \Bigr).
\end{align*}
We proved that for every $x,y\in S$, the expectations $$
\E^x_B[N_y]\le \max_{u\in B}g(u,B) ( 1+ C/h(y)) =:C_y.
$$
Denoting by $\bar \mu_{B_n}=\eta\circ H_{B_n}^{-1}$ the $B_n$-entrance measure of $\eta$ we get that
$$
(\eta\circ (\pi_{B_n})^{-1}) \otimes  \Gamma^{\mathrm{deco}} = \int  \P_B^x(\,\cdot\,)\,d\bar\mu_{B_n}(x)
$$
so that taking monotone limits gives as $n\to\infty$
$$
\int N_y\, d\eta\otimes \Gamma^\mathrm{deco}=\lim_{n\to\infty} \int  \E_B^x[N_y]\,d\bar\mu_{B_n}(x) \le C_y\, \eta(D^*).
$$
Hence, $\eta\otimes \Gamma^\mathrm{deco}$ produces, almost surely, trees in 
 ${\mathfrak T}_B$.
\end{proof}

We give two corollaries. One for the sub-critical case where $Q$ is sub-Markovian and translation invariant and a second one for the symmetric case.

\begin{corollary}Let $B$ be a norming region. Suppose that $(S,\circ)$ is a group and that the operator 
 $Q$ is  a sub-Markovian-kernel on the group $(S,\circ)$ that is translation invariant, i.e., for $x,y,z\in S$, one has $q_{x,y}=q_{x\circ z,y\circ z}$, with $m_z\equiv m<1$.
If there exists $C<\infty$ with  $\bar m_z\le C$ for all $z\in S$, then the constant in~(\ref{cond:deco}) is finite  and all $P_h$-quasi-processes are decorable.
\end{corollary}

\begin{proof}
Recall that by estimate~(\ref{eq8956}), one has
$$
g(x,z) h(z)\leq \max_{u\in B} g(u,B)\, h(x)
$$ 
so that
$$
\frac 1{h(x)}\sum_{z\in B^c} \frac{\bar m_z}{m_z} g(x,z) h(z)^2\le \frac{C}m \,\max_{u\in B} g(u,B)\sum_{z\in B^c} h(z). 
$$
We get validity of (\ref{cond:deco}) by observing that
$$
\sum_{z\in B^c} h(z) \le \sum_{y\in B} \sum_{z\in S}g (z,y)=\#B \sum_{z\in S} g(0,z)=\#B\,(1-m)^{-1}<\infty.
$$
\end{proof}

\begin{corollary}\label{cor:help}
Suppose  that the operator 
 $Q$ is  symmetric and that $\sup_{z\in S}\bar m_z/m_z<\infty$.
 If one of the following two conditions is satisfied for finite constants $C,\eps>0$ and a summable sequence  $(a_k)_{k\in\N_0}$ and 
$\rho:S\to\N_0$
\begin{enumerate}\item $\displaystyle
h(z)^2 \le C\sum_{k=0}^\infty a_k (Q^k)_{0,z}$, \, for all $z\in S$, or

\item $(a_k)$ is decreasing and  $\displaystyle \sum_{k=0}^{\rho(z)} (Q^k)_{0,z}\ge \eps \,h(z)\text{  and  } h(z)\le C a_{\rho(z)}$, \,  for all $z\in S,$
\end{enumerate}
then the constant   in~(\ref{cond:deco}) is finite and all $P_h$-quasi-processes are decorable.
\end{corollary}

\begin{proof}1) We verify that condition 1 entails~(\ref{cond:deco}).
Without loss of generality we can assume  that $(a_k)$ sums up to one. Then for $\tilde h(z)=\sum_{k=0}^\infty a_k (Q^k)_{0,z}$ one has
$$
\sum_{z\in S} \tilde h(z) g(z,x) =\sum_{k=0}^\infty a_k \underbrace{\sum_{z\in S} (Q^k)_{0,z}g(z,x)}_{\le g(0,x)}\le g(0,x).
$$
By symmetry of $Q$, $g(0,x)=g(x,0)=h(x) g(0,0)$ so that
$$
h(x)\ge g(0,0)^{-1}  \sum_{z\in S} \tilde h(z) g(z,x)\ge (C\,g(0,0))^{-1} \sum_{z\in S} h(z)^2 g(z,x)
$$
which implies~(\ref{cond:deco}) since  $\bar m_z/m_z$ is uniformly bounded by assumption.

2) We show that condition 2 entails condition 1. By monotonicity of $(a_k)$ we get that
$$
\sum_{k=0}^\infty a_k (Q^k)_{0,z}\ge a_{\rho(z)} \sum_{k=0}^{a_{\rho(z)}}(Q^k)_{0,z} \ge \eps a_{\rho(z)} h(z)\ge \frac \eps C h(z)^2.
$$
\end{proof}

\begin{theorem}\label{thm:deco2}Suppose   that the operator 
 $Q$ is  symmetric and that $\sup_{z\in S}\frac{\bar m_z}{m_z}<\infty$. If
 $$
\sum_{k=1}^\infty k\, \sup_{z\in S}\, (Q^{k})_{0,z}<\infty,
 $$
 then the constant in~(\ref{cond:deco}) is finite and all $P_h$-quasi-processes are decorable.
\end{theorem}

\begin{proof}We verify criterion 2 of Corollary~\ref{cor:help}.
For $z\in S$ and  $\rho\in\N_0$
$$
\sum_{m=\rho}^\infty (Q^m)_{0,z} \le \sum_{m=\rho}^\infty \sup_{z'\in S}\,(Q^{m})_{0,z'} =: a_{\rho}.
$$
We choose  $\rho(z)$ as the smallest $\rho\in\N_0$ with
$$
a_{\rho+1}\le \frac{h(z)}{2}.
$$
Then 
$$\sum_{m=0}^{\rho(z)} (Q^m)_{0,z}= g(0,z)- \sum_{m=\rho(z)+1}^\infty(Q^m)_{0,z}  \ge h(z) - a_{\rho(z)+1}\ge \frac {h(z)} 2.$$ If $\rho(z)\ge1$, then  $h(z)\le 2 a_{\rho(z)}$, and if $\rho(z)=1$, then $h(z)\le 1\le a_{\rho(z)}$. By assumption, the sum
$$ \sum_{\rho=0} ^\infty a_k = \sum_{\rho=0}^\infty \sum_{k=\rho}^\infty  \sup_{z\in S} (Q^k)_{0,z} =  \sum_{k=0}^\infty (k+1) \sup_{z\in S} (Q^k)_{0,z}  $$
is finite  and hence the result follows with criterion 2 of Corollary~\ref{cor:help}.
\end{proof}

\begin{remark}[The critical  symmetric random walk] Suppose that $Q$ is symmetric and $\sup_{z\in S} \frac{\bar m_z}{m_z}<\infty$. Then if for a finite $C$ and $\delta>0$, for all $k\ge 1$
$$
\sup_{z\in S} (Q^k)_{0,z} \le C k^{-2} (\log (k+1))^{-(1+\delta)},
$$
then the constant in~(\ref{cond:deco}) is finite and  every $P_h$-quasi-process is decorable as consequence of Theorem~\ref{thm:deco2}. As an example one may consider symmetric critical branching random walks on $\Z^d$, where the random walk is  in the domain of attraction of a $\alpha$-stable distribution with $\alpha\in (0,2]$. Typically, in this case for the transition probabilities $P$ one has
$$
\sup_{z\in\Z^d} (P^k)_{0,z}\approx (k+1)^{-d/\alpha},
$$
with possibly additional multiplicative terms of lower order, see \cite[Ch. 9, §50]{GneKol54}. In this case, a critical branching random walk with square-integrable offspring distribution is decorable  if $d>2\alpha$. In particular, we have for $\alpha=2$ that we have decorability if $d>4$. The latter case is also treated under slightly different assumptions in \cite{Zhu2}.
\end{remark}

\begin{remark}[The subcritical symmetric case is decorable]\label{rem946567}
Under the assumptions of Theorem~\ref{thm:deco2} the summability assumption is satisfied if
\begin{align}\label{eq74563}
\sum_{k=1}^\infty  k \sqrt{(Q^{2k})_{0,0}}<\infty.
\end{align}
Indeed, this is the case since for $z\in S$ and $k\in\N$, $Q^{2k}_{0,0}\ge (Q^k)_{0,z} (Q^k)_{z,0}= (Q^k)_{0,z}^2$. This entails that the symmetric subcritical case is always decorable in the following sense. For $\beta\ge0$ we let $Q_\beta=\beta Q$ and
$$
\beta_\mathrm{crit}= \sup\Bigl\{ \beta\geq 0:  \sum_{n=0}^\infty (Q_\beta^n)_{0,0}<\infty\Bigr\}.
$$ 
Then for  every $\beta\in[0,\beta_\mathrm{crit})$, $Q_\beta$ satisfies~(\ref{eq74563}).
\end{remark}

\begin{appendix}
\section{Discrete time quasi-processes}
\label{DiscreteTimeQP}

We briefly  recall the classical concept of a quasi-process. In particular, we introduce the notation that is used for classical quasi-processes in this article.
We suppose that $\bar{S}=S\cup\set{\partial}$ is
the one-point compactification of a discrete set $S$. Moreover, we suppose that $(p_{x,y})_{x,y\in S}$ is a substochastic kernel on $S$ and $(P^x:x\in S)$ the related Markov family on the set of $S$-valued paths
$$
D_+=\bigl\{(w_n)_{n\in\N_0} \in \bar S^{\N_0}: \exists t^\partial\in\N\cup\{\infty\} \text{ such that } [w_n\in S \Leftrightarrow n<t^\partial]\bigr\},
$$  
where $\partial$ is conceived as cemetery state. Moreover, we let
\begin{align*}
D=\bigl\{(w_n)_{n\in\Z}\in \bar S^{\Z}:\  \lim_{n\to-\infty} w_n=\partial \text{ and } &\exists t_\partial,t^\partial\in \Z\cup\{\pm\infty\}\text{ with } t_\partial<t^\partial \\&\quad \text{ such that }[w_n\in S \Leftrightarrow t_\partial \le n<t^\partial]\bigr\}.
\end{align*}
We identify two paths $(w_n),(\bar w_n)\in D$ via an equivalence relation $\sim$, if there exists a time shift $\theta\in \Z$ such that
$$
\forall n\in \Z: w_n=\bar w_{n+\theta}.
$$
Obviously, $\sim$ defines an equivalence relation on $D$ and we denote by $$[\,\cdot\,]
:D\to D^*, (w_n)\mapsto [(w_n)]$$
the canonical embedding of $D$ into the set of equivalence classes $D^*=D/\!\!\sim$.

The set $D_+$ is embedded into $D$ by associating a path $(w_n)_{n\in\N_0}\in D_+$ with the path $(\bar w_n)_{n\in\Z}\in D$ given by
$$
\bar w_n =\begin{cases} w_n, &\text{ if }n\ge0,\\ \partial,& \text{ else}.\end{cases}
$$
Then the image of $D_+$ under $[\,\cdot\,]$ is denoted by $D_+^*$.

For a finite set $B\subset S$ we let $D_B\subset D$ be the set of paths that hit $B$. Note that by finiteness of $B$ and definition of $D$ the following operations are well-defined mappings on $D_B$
\begin{enumerate}
\item $T_B:D_B\to \Z$, $T_B(w)=\min\{n\in \Z: w_n\in B\}$ (entrance time)
\item $H_B:D_B\to S$, $H_B(w)=w_{T_B(w)}$ (entrance point)
\item $\pi_B:D_B \to D_+$, $\pi_B(w)= (w_{T_B(w)+n})_{n\in\N_0}$ ($B$-entrance).
\end{enumerate}
Note that $H_B$ and $\pi_B$ are invariant under time shifts so that $H_B$ and $\pi_B$ are well-defined operations on $D^*_B$ being the class of equivalence classes of all paths hitting $B$. 

$D$ and $D_+$  are endowed with the $\sigma$-field of pointwise evaluation and $D^*$ is the largest $\sigma$-field that makes the embedding $[\,\cdot\,]$ a measurable mapping, see \cite{Szn10} for more details.

\begin{definition}A measure $\eta$ on $D^*$ is called $P$-quasi-process, if for all finite sets $B\subset S$, $\eta(D^*_B)<\infty$ and 
$$
\eta|_{D_{B*}}(\pi_B\in \,\cdot\,) = P^{\bar \mu_B},
$$ 
where $\bar \mu_B= \eta|_{D_{B*}}(H_B\in\,\cdot\,)$.
\end{definition}

One can characterise  the cone of all quasi-processes as follows.

\begin{theorem} Suppose that the associated Green's function is finite, i.e.,
 for all $x,y\in S$ 
$$
g(x,y):=E^x \Bigl[\sum_{n=0}^\infty \1_{\{y\}}(X_n)\Bigr]=\sum_{n=0}^\infty (P^n)_{x,y}<\infty.
$$
Then there exists an isomorphism between the cone of all quasi-processes $\eta$ and the cone of all $P$-excessive measures $\nu$ so that for the related pair $\eta$ and $\nu$
$$
\nu(x)= \int  \sum_{n\in\Z} \1_{\{x\}}(w_n)\, d\eta([w])\qquad (x\in S).
$$
\end{theorem}

\end{appendix}

{\bf Acknowledgement.}
Funded by the Deutsche Forschungsgemeinschaft (DFG, German Research Foundation) under Germany's Excellence Strategy EXC 2044--390685587, Mathematics Münster: Dynamics--Geometry--Structure.

\bibliography{biblio}
\bibliographystyle{plain}
\end{document}